\UseRawInputEncoding
\documentclass[12pt]{siamart190516} 
\usepackage{textcomp}
\usepackage{graphicx,makeidx}
\usepackage{epsf}
\usepackage{amsmath,amsfonts,amssymb,latexsym,bm}
\usepackage{enumitem}
\usepackage{setspace}
\usepackage{mathtools}
\usepackage{lineno}

\newtheorem{claim}{Claim}

\newtheorem{remark}[theorem]{Remark}
\newtheorem{observation}[theorem]{Observation}
\newtheorem{notation}[theorem]{Notation}

\newcommand{\ignore}[1]{}

\newcommand{\remove}[1]{}

\newcommand{\janthirteen}[1]{#1}
\newcommand{\janseven}[1]{#1}
\newcommand{\janonerem}[1]{}
\newcommand{\janone}[1]{#1}
\newcommand{\deceight}[1]{#1}
\newcommand{\octthirteen}[1]{#1}
\newcommand{\octeight}[1]{#1}
\newcommand{\octone}[1]{#1}
\newcommand{\maythreeone}[1]{#1}
\newcommand{\maytwonine}[1]{#1}
\newcommand{\mayoneseven}[1]{#1}
\newcommand{\mayonefour}[1]{#1}
\newcommand{\mayoneone}[1]{#1}
\newcommand{\maysix}[1]{#1}
\newcommand{\mayseven}[1]{#1}
\newcommand{\augfifteen}[1]{#1}
\newcommand{\julythirteen}[1]{#1}
\newcommand{\juneonenine}[1]{#1}
\newcommand{\junesix}[1]{#1}
\newcommand{\sept}[1]{#1}
\newcommand{\augusttwenty}[1]{#1}
\newcommand{\augustsix}[1]{#1}
\newcommand{\mayfour}[1]{#1}
\newcommand{\aprilEight}[1]{#1}
\newcommand{\marchTwentysix}[1]{#1}
\newcommand{\marchTen}[1]{#1}
\newcommand{\febSixteenChange}[1]{#1}
\newcommand{\febEightChange}[1]{#1}
\newcommand{\febFiveChange}[1]{#1}
\newcommand{\greenchange}[1]{#1}
\newcommand{\redchange}[1]{#1}
\newcommand{\bluechange}[1]{#1}

\newcommand{\janThree}[1]{#1}
\newcommand{\decTwoOhTwoOh}[1]{#1}

\newcommand{\rbr}{\bluechange}

\newcommand{\janFourteen}[1]{#1}

\newcommand{\febTwelf}[1]{#1}

\usepackage{accents} 

\title{Extending drawings of complete graphs into arrangements of pseudocircles}
\author{Alan Arroyo\footnote{\octthirteen{Supported by CONACyT.  This project has received funding from the European UnionÕs Horizon 2020 research and innovation programme under the Marie Sk\l{}odowska-Curie grant agreement No.\ 754411, IST, Klosterneuburg, Austria.}}, R.\ Bruce Richter\footnote{Supported by NSERC Grant 50503-10940-500,  Dept.\ of Combinatorics \& Optimization, University of Waterloo, Waterloo, Canada N2L 3G1.}, and Matthew Sunohara\footnote{Supported by NSERC, Dept.\ of Mathematics, University of Toronto, Toronto, Canada M5S 1A1.
 }\\ {\scriptsize\texttt{alanarroyoguevara@gmail.com, brichter@uwaterloo.ca, matthew.sunohara@mail.utoronto.ca}}}

\date{\LaTeX-ed: \today}

\newcommand{\startClaims}{\setcounter{claim}{0}}

\newenvironment{cproof}%
{\noindent{\bf Proof.}\startClaims\  }%
{\hfill\eop\par\bigskip}%

\newenvironment{proofof}[1]%
{\noindent{\bf Proof of #1.}}%
{\hfill\eopf\par\bigskip}%

\def\eop{\hfill{{\rule[0ex]{7pt}{7pt}}}}

\newenvironment{cproofof}[1]
{\bigskip\noindent{\bf Proof of #1.}\startClaims\ }
{\hfill{\eop}\par\bigskip}

\renewenvironment{proof}
{\noindent{\bf Proof.}}
{\hfill{$\Box$}\par\bigskip}

\newcommand{\eopf}{\raisebox{0.8ex}{\framebox{}}}

\def\meet#1#2{\times_{#1^-\kern-3pt,\hskip1pt #2^+}}
\def\sphere{{\mathbb S^2}}

\newcommand{\apriltwonine}[1]{#1}

\newcommand{\jantwoseven}[1]{#1}

\begin{document}

%
\maketitle              
%




\begin{abstract}
Motivated by the successful application of geometry to proving the Harary-Hill Conjecture for ``pseudolinear" drawings of $K_n$, we introduce ``pseudospherical" drawings of graphs.
\octthirteen{A {\em spherical drawing} 
 of a graph $G$ is a drawing in the unit sphere $\mathbb{S}^2$ in which the vertices of $G$ are represented as points---no three on a great circle---and the edges of $G$ are shortest-arcs in $\mathbb{S}^2$
 connecting pairs of vertices.  Such a drawing has three properties}:  \decTwoOhTwoOh{(1)} every edge $e$ is contained in a simple closed curve $\gamma_e$  such that the only vertices in $\gamma_e$ are the ends of $e$; \decTwoOhTwoOh{(2)} if $e\ne f$, {then $\gamma_e\cap\gamma_f$ has precisely two crossings}; and \decTwoOhTwoOh{(3)} if $e\ne f$, then $e$ intersects $\gamma_f$ at most once\deceight{, either \decTwoOhTwoOh{in} a crossing or an end of $e$}.  We use \decTwoOhTwoOh{Properties (1)--(3)}  to define a {\em pseudospherical drawing\/} of $G$
.
Our main result is that, for the complete graph, \decTwoOhTwoOh{Properties (1)--(3)}  are equivalent to the same three properties but with ``{\decTwoOhTwoOh{precisely} two \octthirteen{crossings}}" \decTwoOhTwoOh{in (2)} replaced by ``{at most two crossings}".   


\janthirteen{The proof requires a result in the geometric transversal theory of arrangements of pseudocircles. This is proved using the} surprising result that the absence of special arcs ({\em coherent spirals}) in an arrangement of simple closed curves characterizes the fact that any two curves in the arrangement have at most \octthirteen{two} crossings.

Our studies provide the necessary ideas for exhibiting a drawing of $K_{10}$ that has no extension to an arrangement of pseudocircles and a drawing of $K_9$ that does extend to an arrangement of pseudocircles, but no such extension has all pairs of pseudocircles crossing twice.
\end{abstract}

\section{Introduction}\label{sec:intro}

The Harary-Hill Conjecture states that the crossing number of the complete graph $K_n$ is given by the formula 

\[
H(n):=\frac14\left\lfloor \frac {\mathstrut{n}}{\mathstrut{2}}\right\rfloor \left\lfloor \frac {n-1}2\right\rfloor \left\lfloor \frac {n-2}2\right\rfloor  \left\lfloor \frac {n-3}2\right\rfloor \text{.}
\]

\noindent Some of the known families of drawings of $K_n$ achieving $H(n)$ crossings have the geometric character of being spherical: a {\em spherical drawing} 
 of a graph $G$ is a drawing in the unit sphere $\mathbb{S}^2$ in which the vertices of $G$ are represented as points\octthirteen{---no three on a great circle---}and the edges of $G$ are shortest-arcs in $\mathbb{S}^2$
 connecting pairs of vertices.
 
Examples of \jantwoseven{families having $H(n)$ crossings are:}  Hill's Tin Can Drawings \cite{hh}; \octthirteen{Kyn\v cl's general spherical drawing in his posting \cite{kyncl}};  the family of \'Abrego et al.\ \cite{abrego3} in which every edge is crossed at least once; \rbr{and  the crossing-minimal 2-page drawings  in \'Abrego et al.~\cite[Sec.~4.3, 4.4]{abrego2}, where they show that the 2-page crossing number of $K_n$ is $H(n)$.   \jantwoseven{The first three of these are known to be spherical.}}    

\deceight{A surprising result by Moon \cite{moon} states that the expected number of crossings in a random spherical drawing of $K_n$ is $\frac{3}{8}{n \choose 4}$.   Thus,}  spherical drawings are linked with the asymptotic version of the Harary-Hill Conjecture:
$$\lim_{n\rightarrow \infty }\frac{\text{cr}(K_n)}{{n \choose 4}}=\frac{3}{8}\text{.}$$ 

{\jantwoseven{No drawing of $K_n$} having $H(n)$ crossings is known to be non-spherical (at least up to Reidemeister-type moves; see below).}  (For $n$ even, the 2-page drawings in \cite{abrego2} are ``pseudospherical"; see the discussion in our Section \ref{sec:hcxPs}.) However, it is still unknown whether every spherical drawing of $K_n$ has at least $H(n)$ crossings.  

The analogue in the plane of spherical drawings is rectilinear drawings. A {\em rectilinear} drawing of a graph $G$ is a drawing of $G$ \octthirteen{in the plane} {so that its} edges are straight-line segments. One of the most important recent accomplishments in the study of crossing numbers is a result of \'Abrego and Fern\'andez-Merchant \cite{AF} and, simultaneously and independently, Lov\'asz et al.\  \cite{lovasz}, showing that rectilinear drawings of $K_n$ have at least $H(n)$ crossings.  \rbr{(It follows from \cite{agor}, especially Theorem 11 there, that, for $n\ge 10$, rectilinear drawings have strictly more than $H(n)$ crossings.)}  {There is a quite direct line from this early work,  \rbr{via shellability \cite{nineAuthor,aichEtAl}}, to \'Abrego et al \cite{abrego2} determining the 2-page crossing number of $K_n$.} 

The proofs that rectilinear drawings of $K_n$ have at least of $H(n)$ crossings \deceight{use} machinery for studying arrangements of pseudolines, and only require the property that the edges in a rectilinear drawing can be extended to such an arrangement. Drawings whose edges can be extended into an arrangement of pseudolines are called \emph{pseudolinear}. An analogous property is satisfied by spherical drawings: each edge can be extended to a great circle. This implies that the edges can be extended into an {\em arrangement of pseudocircles\/}, defined as a set of simple closed curves in $\mathbb{S}^2$ such that every two intersect at most twice, and every intersection is a crossing between two curves. 

The success of the geometric approach for the rectilinear crossing number of $K_n$ suggests trying an analogous approach for spherical drawings, replacing pseudolinear drawings with a suitable generalization of spherical drawings, which we call ``pseudospherical'' drawings and define below.

There are nice characterizations of pseudolinear drawings of $K_n$.  Aichholzer et al.\  \cite{oswin} prove that a drawing of $K_n$ in the plane is pseudolinear if and only if every crossing $K_4$ has the facial 4-cycle bounding the infinite face of the $K_4$.  Arroyo et al.\  \cite{amrs} have an equivalent characterization: a drawing of $K_n$ is pseudolinear if and only if it is f-convex.  Arroyo et al.\  \cite{abr} characterize when a drawing of a general graph in the plane is pseudolinear.

The notion of an f-convex drawing is introduced by Arroyo et al.~in \cite{convex} as part of the \emph{convexity hierarchy}: 
\[
\{\textrm{convex drawings}\} \supset \{\textrm{h-convex drawings}\} \supset \{\textrm{f-convex drawings}\}.
\]
A drawing $D$ of $K_n$ is {\em convex\/} if, for each 3-cycle $T$ of $K_n$, there is a closed side $\Delta$ of the drawing $D[T]$ of $T$ such that, if vertices $x,y$ are drawn in $\Delta$, then $D[xy]\subseteq \Delta$.  A principal theorem in \cite{amrs2} is a partial result suggesting that all crossing-minimal drawings of $K_n$ might be convex.

\rbr{A convex drawing $D$ of $K_n$ is {\em h-convex\/} if the convex side $\Delta_T$ of each 3-cycle $T$ may be chosen so that, if a 3-cycle $T'$ is, for another 3-cycle $T$, drawn in $\Delta_T$, then $\Delta_{T'}\subseteq \Delta_T$. In unpublished work, the method for the principal theorem in \cite{amrs2} mentioned above shows that a convex, but not h-convex, drawing of $K_n$ is not crossing-minimal. Putting these two ideas together, it is conceivable that every crossing-minimal drawing of $K_n$ is h-convex.}

It is quite easy to show that every spherical drawing is h-convex (see Section \ref{sec:hcxPs}). However, it is not true that every h-convex drawing is spherical. The property of h-convexity is preserved under (the natural analogue of) Reidemeister III moves (see Figure \ref{fg:reid}), but sphericity is not. Reidemeister III moves preserve three properties of spherical drawings which we take as the definition of a pseudospherical drawing.

\begin{figure}[ht!]
\begin{center}
\rbr{\includegraphics[scale=.5]{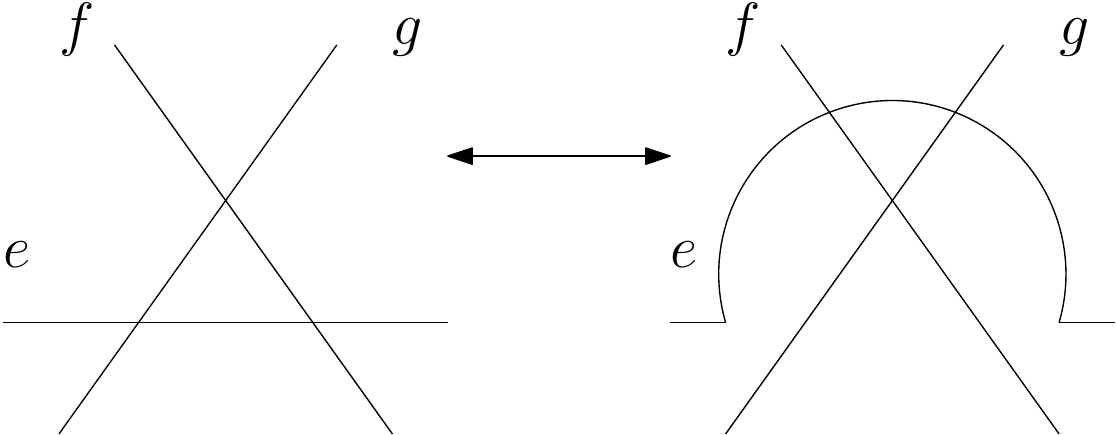}
\caption{Reidemeister III move.}}
\label{fg:reid}
\end{center}
\end{figure}

A drawing of a graph $G$ in the sphere is {\em pseudospherical\/} if\deceight{, for any distinct edges $e$ and $f$}: \begin{enumerate}[leftmargin=60pt,label={\bf (PS\arabic*)}, ref=(PS\arabic*),itemsep=0pt]
\item\label{it:psVertex} \deceight{$e$ is} contained in a simple closed curve $\gamma_e$ such that no vertex other than an end of $e$ is contained in $\gamma_e$;  
\item\label{it:psCrossings} \deceight{$|\gamma_e\cap \gamma_f|=2$} and all intersections are crossings; and
\item\label{it:psEdge} \deceight{$e\cap \gamma_f$} has at most one point.
\end{enumerate}

The least obvious part of the definition is perhaps \ref{it:psEdge}, which can be thought of as a combinatorial analogue of the property that an edge in a spherical drawing is not just a geodesic arc between its endpoints, but a shortest-arc between its endpoints.
 
Since the number of crossings of a drawing is also preserved under Reidemeister III moves, one obtains non-spherical drawings of $K_n$ with $H(n)$ crossings from the spherical examples given above. It is more natural to consider pseudospherical drawings than spherical drawings of $K_n$ in connection with crossing-minimality because pseudosphericity is preserved under Reidemeister III moves.

{The principal part} of our main theorem is to show that a drawing of $K_n$ is pseudospherical if and only if it is h-convex, a property which can be checked in polynomial time (see Section \ref{sec:hcxPs}). This can be seen as a parallel to the \jantwoseven{characterization} by Arroyo et al. \cite{amrs} of pseudolinear drawings of $K_n$ as f-convex drawings of $K_n$. The equivalence also enables the introduction of geometric methods to the study of h-convex drawings and bolsters the idea that all crossing-optimal drawings of $K_n$ are h-convex/pseudospherical: evidence that crossing-optimal drawings of $K_n$ are h-convex suggests that they are pseudospherical and vice versa.

In fact, we will do quite a bit more than \jantwoseven{characterize} pseudospherical drawings of $K_n$. The reader may have wondered about the choice of equality in \ref{it:psCrossings}. A natural variation of the notion of a pseudospherical drawing results from weakening \ref{it:psCrossings}. A drawing $D$ is {\em weakly pseudospherical\/} if it satisfies \ref{it:psVertex}, \ref{it:psEdge}, \deceight{while \ref{it:psCrossings} is relaxed to 
\begin{enumerate}[leftmargin=60pt,label={\bf (PS\arabic*w)}, ref=(PS\arabic*w),itemsep=0pt,start=2]
\item\label{it:psWeakCrossings} $|\gamma_e\cap \gamma_f|\le 2$ and all intersections are crossings.
\end{enumerate}
}
Our main result is the following.

\begin{theorem}\label{th:main}  
For a drawing $D$ of $K_n$, the following are equivalent:
\begin{enumerate}[leftmargin=60pt, label={\bf (\ref{th:main}.\arabic*)},ref=(\ref{th:main}.\arabic*)]
\item\label{it:Dps} $D$ is pseudospherical;
\item\label{it:Dwps} $D$ is weakly pseudospherical; and
\item\label{it:Dhc} $D$ is h-convex.
\end{enumerate}
\end{theorem}

The implication \ref{it:Dps}${}\Rightarrow{}$\ref{it:Dwps} is trivial.  Although the reader may not see it now, the implication \ref{it:Dwps}${}\Rightarrow{}$\ref{it:Dhc} is quite easy to prove.  The hard part is showing \ref{it:Dhc}${}\Rightarrow{}$\ref{it:Dps}.  We do not see how to demonstrate \ref{it:Dwps}${}\Rightarrow{}$\ref{it:Dps} directly.

The proof of \ref{it:Dhc}${}\Rightarrow{}$\ref{it:Dps} proceeds by iteratively finding a curve $\gamma_{e'}$ for the next edge $e'$ to extend by one {the current set $\Gamma$ of $\gamma_e$s} satisfying the conditions \ref{it:psVertex}--\ref{it:psEdge}.  There are two principal steps involved.  The first step is to find \janseven{two initial approximations} for $\gamma_{e'}$ (\janseven{when added to $\Gamma$, either of these} will \deceight{satisfy \ref{it:psWeakCrossings}}), while the second is to \janseven{repeatedly shift one of the initial approximations, gradually increasing the number of curves in \deceight{$\Gamma$} that it intersects until it intersects them all, at which point it is a $\Gamma$-transversal}.

Each of these steps has its challenges.  To find the initial \janseven{approximations}, we require an extensive study of h-convex drawings; this is done in Section \ref{sec:hcxPs}.   Crucially, each edge of an h-convex drawing partitions the vertex set of $K_n$ into two pseudolinear drawings of (typically smaller) complete subgraphs\janseven{; the initial approximations are near the outer boundary of each of these pseudolinear subdrawings}.  (Motivated by this study, we tried to show this partitioning holds for  pseudospherical drawings of general graphs.  In a personal communication, {Xinyu Lily} Wang has shown that it is false  in the more general context.)

\janseven{Producing the $\Gamma$-transversal requires shifting one of a pair of initial approximations from the preceding paragraph towards the other using analogues of Reidemeister II and III moves.}
The core of the shifting turns out to require a characterization of  an arrangement of pseudocircles, which we recall is a set of simple closed curves in the sphere such that every two intersect at most twice, and every intersection is a crossing between the two curves.  

The characterization of arrangement of pseudocircles requires an even more general notion.  An {\em arrangement of simple closed curves\/} in the sphere is a set of simple closed curves, any two of which have finitely many intersections, all of which are crossings. If $\Sigma$ is an arrangement of simple closed curves, then a {\em spiral\/} of $\Sigma$ is an arc \rbr{(that is, a homeomorph of a compact interval)} in the union $P(\Sigma)$ of the curves in $\Sigma$ that always makes the same---either all left or all right---turn in changing from one curve to another \janone{(see Figure \ref{fg:spiral})}.  \deceight{In Section \ref{sec:technical}, we give a more precise definition of spiral and define {\em coherent spirals\/}}. The auxiliary result that we need is the following.

\begin{theorem}\label{th:auxiliary}
Let $\Sigma$ \octthirteen{be an arrangement of simple closed curves}.  Then $\Sigma$ is an arrangement of pseudocircles if and only if $\Sigma$ has no coherent spirals.
\end{theorem}

Our \janone{study of spirals} led us to two drawings, one for each of $K_{10}$ and $K_9$.  The former has no extension of its edges to an arrangement of pseudocircles.  The latter has such an extension, but no extension has all pseudocircles crossing exactly twice.  These examples are exhibited in Section \ref{sec:K9andK10}.    

\deceight{Independently, Aichholzer at the \julythirteen{2015 Crossing Number Workshop in Rio de Janeiro} and  Pilz  at \janone{the} Geometric Graph Week in Berlin (2015) asked if every drawing of $K_n$ has an extension to an arrangement of pseudocircles.   The drawing of $K_{10}$ answers this question in the negative.  The drawing of $K_9$ answers negatively the related question of improving a pseudocircular extension to a pseudospherical extension.}

\octthirteen{\deceight{In the next section, we introduce spirals and prove} Theorem \ref{th:auxiliary}.  \janseven{Section~\ref{sec:transversal} contains the proof that an initial pair of approximations for $\Gamma$ implies the existence of a $\Gamma$-transversal}. 
Section \ref{sec:hcxPs} introduces h-convex drawings and proves the easy implication \ref{it:Dwps}${}\Rightarrow{}$\ref{it:Dhc}.  Section~\ref{sec:HcxIsPS} contains the necessary discussion of h-convex drawings to obtain the initial \janseven{approximations}, which is done in Section \ref{sec:hCxHasExactExtension}, thereby completing the proof of Theorem \ref{th:main}.    The interesting drawings of $K_9$ and $K_{10}$ are in Section \ref{sec:K9andK10}\deceight{, while Section \ref{sec:conclusion} has concluding remarks.}}

\section{Spirals and Coherence}\label{sec:technical}

\mayoneone{A spiral is a special arc}, \mayonefour{illustrated in Figure \ref{fg:spiral} and} \mayoneone{defined precisely below, in the union of \janseven{an arrangement of} simple closed curves that has all its crossings facing the same side of the arc.} {(Alternatively, a spiral always makes the same---either all left or all right---turn \decTwoOhTwoOh{when it changes} from one curve to another; \decTwoOhTwoOh{it is permitted to continue straight through a crossing}.)}    The main result in this section is \mayseven{a characterization of arrangements of pseudocircles as {either (a) not having any ``coherent" spirals or, equivalently,} (b) every spiral has an ``external segment".  That an arrangement of pseudocircles has no coherent spiral is the point required for the proof of Theorem \ref{th:main}.}

\begin{figure}[ht!]
\begin{center}
\includegraphics[scale=1]{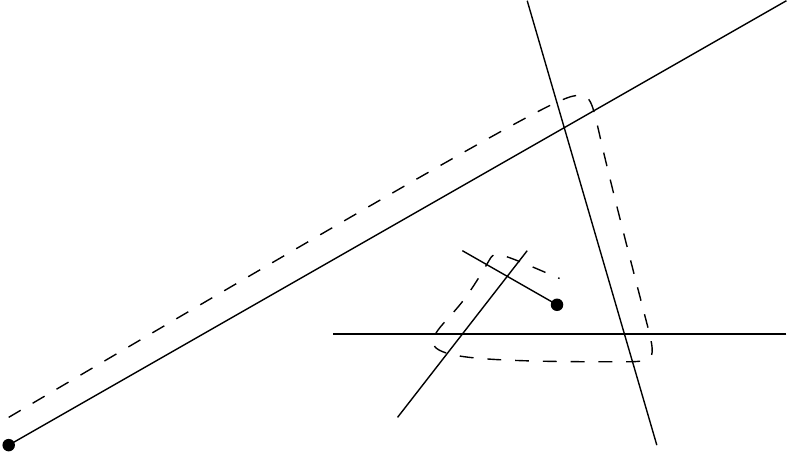}
\caption{An example of a spiral.}
\label{fg:spiral}
\end{center}
\end{figure}

\mayoneone{\octthirteen{In an arrangement of simple closed curves $\Gamma$ \deceight{(defined just before Theorem \ref{th:auxiliary})},} three or more of the  curves in $\Gamma$ may cross at the same point.  This is necessary for extensions of drawings of $K_n$, where $n-1$ curves pairwise cross at each vertex.}

\mayseven{A non-trivial arc \decTwoOhTwoOh{(that is, not just a single point)} $A$ \octthirteen{in the union $P(\Gamma)$ of the curves in $\Gamma$}} with ends $s$ and $t$ has a unique decomposition sequence $\alpha_0\alpha_1\dots\alpha_m$ of subarcs of $A$\mayonefour{, as depicted in Figure \ref{fg:arcExample},} such that: \begin{enumerate}[label={(\roman*)},leftmargin = .50 truein,ref={(\roman*)}]
\item $s$ is an end of $\alpha_0$, $t$ is an end of $\alpha_m$; 
\item for each $i=0,1,\dots,m$, there is a $\gamma_i\in \Gamma$ such that $\alpha_i\subseteq \gamma_i$; and 
\item\label{it:distinct} for $i=1,2,\dots,m$,  the curves $\gamma_{i-1}$ and $\gamma_i$ in $\Gamma$ are distinct and $\alpha_{i-1}\cap \alpha_i$ is a crossing of $\gamma_{i-1}$ and $\gamma_i$.
\end{enumerate} 
The number $m$ is the {\em weight of $A$\/}.  \mayseven{Figure \ref{fg:arcExample} shows an arrangement of simple closed curves with an arc of weight 3.}
For $i=1,2,\dots,m$, the crossing of $\alpha_{i-1}$ with $\alpha_i$ is denoted $\times_i$.  For convenience, we set $\times_0=s$ and $\times_{m+1}=t$.


\begin{figure}[ht!]
\begin{center}
\includegraphics[scale=.7]{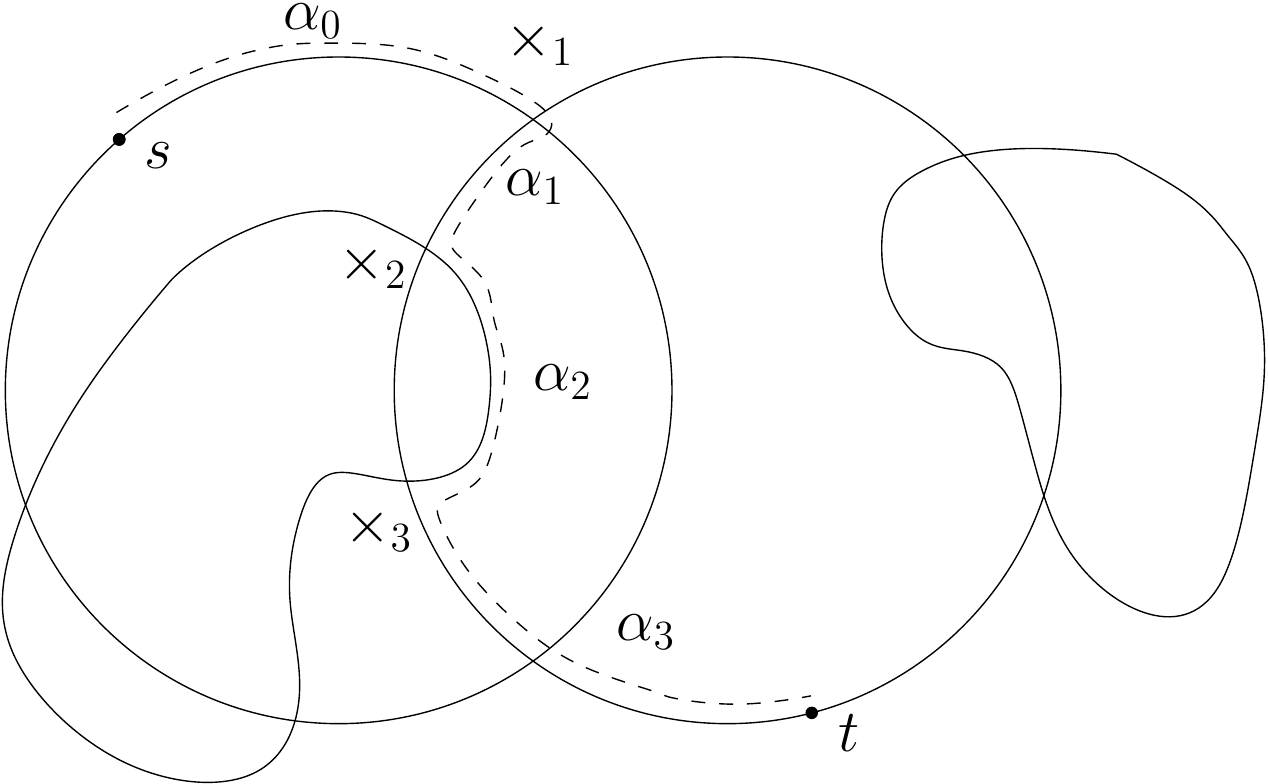}
\caption{\apriltwonine{\decTwoOhTwoOh{Illustrated is an arrangement of four simple closed curves.  The $st$-arc indicated by the dotted curve has decomposition sequence $\alpha_0\alpha_1\alpha_2\alpha_3$.  Notice that the $\times_3t$-arc $\alpha_3$ makes no turn at the crossing in its middle and the right-hand circle is both $\gamma_1$ and $\gamma_3$.}}}
\label{fg:arcExample}
\end{center}
\end{figure}

For $i=0,1,\dots,m$, $\alpha^-_{i}$ and $\alpha^+_i$ are the closures of the components of $\gamma_{i}\setminus A$ incident with $\times _i$ and $\times_{i+1}$, respectively.  Four such $\alpha^{\varepsilon}_j$ are illustrated in Figure \ref{fg:arcDecomposition}.   \mayseven{Evidently,  $\alpha_i^\pm$ consists of the continuation of $\alpha_i$ through either $\times_i$ ($-$) or $\times_{i+1}$ ($+$) up to the next meeting with $A$.  }

The continuations $\alpha_i^+$ and $\alpha_{i+1}^-$ both leave $\times_{i+1}$ on the same side {of} $A$.  This is the {\em side of $A$ that $\times_{i+1}$ faces\/}.  \mayonefour{In Figure \ref{fg:arcExample}, $\times_1$ and $\times_2$ face different sides of the dotted arc.}  We are now prepared for the definitions of \mayseven{spiral, external segment, and coherence.}

\begin{definition}[Spiral, External Segment, Coherence]\label{df:coherence} 
Let $\Gamma$ be an arrangement of simple closed curves and 
let $A$ be an arc in $P(\Gamma)$ with decomposition sequence $\alpha_0\alpha_1\dots$ $\alpha_m$.  For each $i=0,1,2\dots,m$, let $\gamma_i\in \Gamma$ be such that $\alpha_i\subseteq \gamma_i$.  (Only consecutive $\gamma_i$ are required to be distinct; if $j>i+1$, then $\gamma_{j}$ could be the same as $\gamma_i$.)

\begin{enumerate}[label=(\ref{df:coherence}.\arabic*),leftmargin = 60pt,ref=(\ref{df:coherence}.\arabic*)] 
\item The arc $A$ is a {\em spiral\/} if all of $\times_1,\dots,\times_m$ face the same side of $A$.
\item For $i\in\{0,1,2,\dots,m\}$, the segment $\alpha_i$ is \[
\left\{\begin{array}{ll}\textrm{{\em external\/} for $A$},  &\gamma_i\cap A=\alpha_i\\ \textrm{{\em internal\/} for $A$,}& \textrm{otherwise}\end{array}\right.\,.
\]
\item For $i\in\{0,1,2,\dots,m\}$ and $\varepsilon\in \{+,-\}$, the arc $\alpha^\varepsilon_i$ is a {\em coherent extension\/} if\maytwonine{: (a)} $\alpha_i$ is internal for $A$ and \maytwonine{(b)} $\alpha^\varepsilon_i$ has both ends on the same side (see discussion of ``side" below) of the interior of $A$.  \maytwonine{(In Figure \ref{fg:arcDecomposition}}, $\alpha^-_1$ is not a coherent extension but $\alpha_0^+$, $\alpha_3^-$, and $\alpha_3^+$ are.)

\item\label{it:coherent} For $i\in\{0,1,2,\dots,m\}$, the segment $\alpha_i$ is {\em coherent\/} if \redchange{at least one of $\alpha^-_i$ and $\alpha^+_i$ is a  coherent extension}.  In Figure \ref{fg:arcDecomposition}, both $\alpha_0$ and $\alpha_3$ are coherent.

\item\label{it:Acoherent}  The arc $A$ is {\em coherent\/} if, for each $i=0,1,2,\dots,m$,  $\alpha_i$ is coherent.
\end{enumerate}
\end{definition}

\begin{figure}
\begin{center}
\includegraphics[scale=.875]{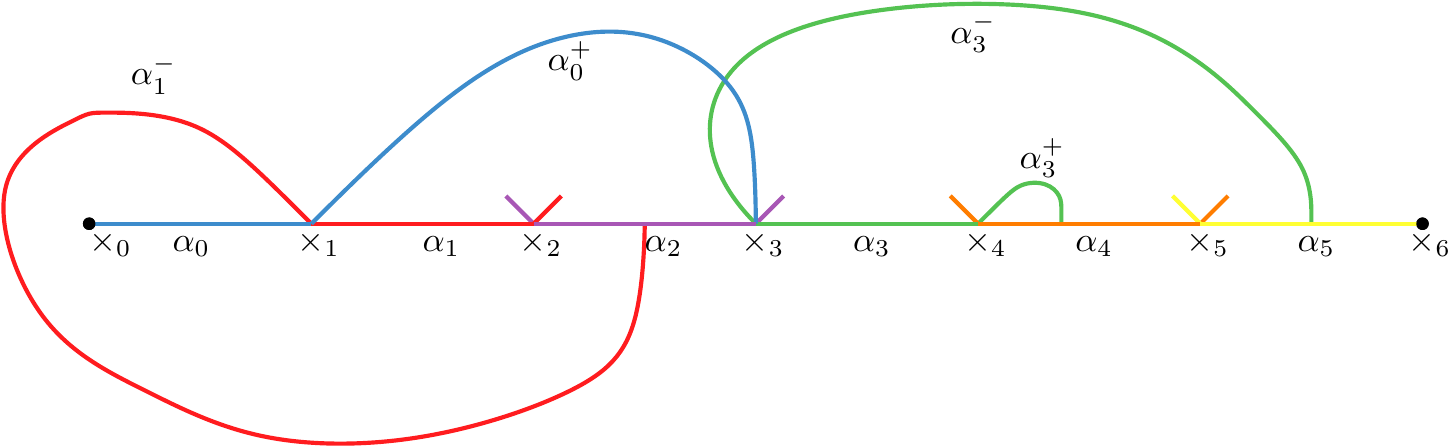}
\caption{Illustrating several points in Definition \ref{df:coherence}.  The arc $A$ is homeomorphic to a straight line segment and is so represented in the diagram.}
\label{fg:arcDecomposition}
\end{center}
\end{figure}

As we traverse $A$ from one end to the other, there are naturally left and right sides.  The two ends of the arc $\alpha^\pm_i$ are in $A$.  The issue in the definition of coherence is: are the points near each end of $\alpha^\pm_i$ on the same side of $A$ or not. ``Left" and ``right" depend on an orientation of $A$ and are irrelevant to us.

 \greenchange{To be on the same side, the ends of $\alpha^\pm_i$ must be in the interior of $A$.  As the points $\times_0$ and $\times_{m+1}$ are not in the interior of $A$, we have the following.}

 \begin{remark}\label{rk:alpha0-NotCoherent} \greenchange{Neither $\alpha_0^-$ nor  $\alpha_m^+$ is coherent.}\end{remark}  

\mayseven{The following characterization of \octthirteen{an arrangement} of pseudocircles in terms of its spirals seems to be quite interesting in its own right.  We only need \ref{it:arrpseudo}${}\Rightarrow{}$\ref{it:noCoherSpiral} for the proof of Theorem \ref{th:main}.}

\begin{theorem}\label{th:technical}  Let $\Gamma$ be an arrangement of simple closed curves in the sphere.  \mayseven{Then the following are equivalent:
\begin{enumerate} [label={\bf(\ref{th:technical}.\arabic*)}, ref=(\ref{th:technical}.\arabic*), leftmargin=60pt]
\item\label{it:arrpseudo} $\Gamma$ is an arrangement of pseudocircles;

\item\label{it:noCoherSpiral} $P(\Gamma)$ has no coherent spirals; 

\item\label{it:noWt1CoherSpiral} $P(\Gamma)$ has no coherent spirals with weight 1; and
\item\label{it:SpiralsAllExternal} every spiral $A$ in $P(\Gamma)$ has a segment external for $A$.
\end{enumerate}
}
 \end{theorem}
 
\octthirteen{This characterization fits in well with much recent work on arrangements of pseudocircles.  The paper of Felsner and Scheucher \cite{fs} is one example; their references [2], [7], [10], [15], [19], and [26] are others.  Felsner and Scheucher have a web page devoted to pseudocircles \cite{fsweb}.  Another recent work about unavoidable configurations in the sense of Ramsey's Theorem  is Medina et al.~\cite{mrs}.}

The example {to the left in} Figure \ref{fg:counterexample} \mayseven{has a spiral $A$ in an arrangement of pseudocircles}.   \mayseven{This spiral is not \octthirteen{coherent because} $\alpha_1$ is external for $A$.}

\begin{figure}[ht!]
\includegraphics[scale=1]{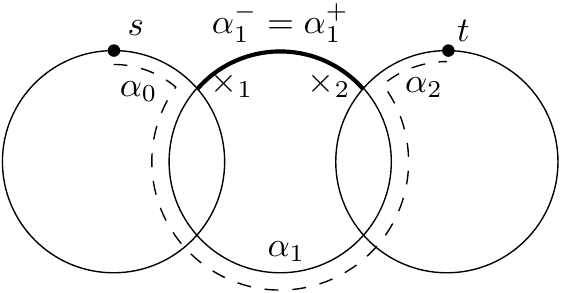} \hfill \includegraphics[scale=1]{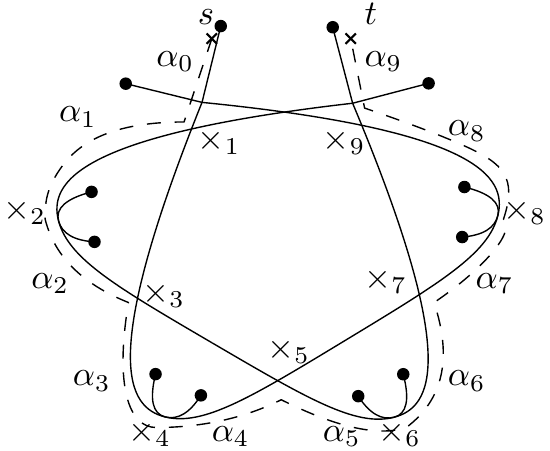}
\caption{\octthirteen{The dotted arcs are spirals. The left-hand one has weight 2 and has an external segment. The right-hand one is coherent and is in a drawing related to the examples in Section \ref{sec:K9andK10}.}   
}
\label{fg:counterexample}
\end{figure}

{In addition to its use in the proof of Theorem \ref{th:main}, Theorem \ref{th:technical}  is useful for constructing drawings of graphs that cannot be extended to arrangements of pseudocircles. For example,  the right-hand diagram in Figure \ref{fg:counterexample} cannot be extended to an arrangement of pseudocircles because the $st$-arc indicated by the dashed curve induces a coherent spiral in any such extension. A similar idea is used in Section \ref{sec:K9andK10} to construct a non-extendible drawing of $K_{10}$.}

\begin{cproofof}
{Theorem \ref{th:technical}}
\mayoneone{\decTwoOhTwoOh{The implication \ref{it:noCoherSpiral}${}\Rightarrow{}$\ref{it:noWt1CoherSpiral} is trivial}.  We prove \ref{it:noWt1CoherSpiral}${}\Rightarrow{}$\ref{it:arrpseudo} and \ref{it:arrpseudo}${}\Rightarrow{}$\ref{it:noCoherSpiral} to complete the proof that \ref{it:arrpseudo}, \ref{it:noWt1CoherSpiral}, and \ref{it:noCoherSpiral} are equivalent.  \decTwoOhTwoOh{Since \ref{it:SpiralsAllExternal}${}\Rightarrow{}$\ref{it:noCoherSpiral} is trivial, we} finish the proof with \ref{it:noCoherSpiral}${}\Rightarrow{}$\ref{it:SpiralsAllExternal}.}

\mayseven{\ref{it:noWt1CoherSpiral}${}\Rightarrow{}$\ref{it:arrpseudo}.  Suppose by way of contradiction that $\gamma_0,\gamma_1$ are distinct curves in $\Gamma$ such that $|\gamma_0\cap \gamma_1|>2$.  Then $|\gamma_0\cap \gamma_1|\ge 4$.  Let $s\in \gamma_0\setminus \gamma_1$ and, traversing $\gamma_0$ in one direction starting at $s$, let $p_1,p_2,p_3$ be the first three points of $\gamma_0\cap \gamma_1$ encountered.}

\mayoneone{Let $A$ be the arc obtained by starting at $s$, continuing along $\gamma_0$ through $p_1$ to $p_2$ and then following $\gamma_1\setminus\{p_1\}$ through $p_3$ to a point $t$ just beyond $p_3$.  The decomposition of $A$ is $\alpha_0\alpha_1$, with, for $i=1,2$, $\alpha_i\subsetneq A\cap \gamma_i$.  Since every arc in $P(\Gamma)$ with weight at most 1 is a spiral, $A$ is a spiral.}
\maytwonine{The fact that the points $p_2$ and $p_3$ of $\gamma_0\cap \gamma_1$ are consecutive in $\gamma_0$ imply that $\alpha_0^+$ is coherent.}

\mayoneseven{
On the other hand, the $sp_2$-subarc of $A$ intersects the $p_1p_2$-subarc of $\gamma_1\setminus\{p_3\}$ just in $\{p_1,p_2\}$.  This shows that $\alpha_1^-$ is coherent, completing the proof that $A$ is coherent, the required contradiction.
}

  \mayseven{\ref{it:arrpseudo}${}\Rightarrow{}$\ref{it:noCoherSpiral}.  To obtain a contradiction, suppose $A$ is a coherent spiral with least weight.  An arc with weight 0 is not coherent, so the decomposition  $\alpha_0\alpha_1\cdots\alpha_m$ of $A$ has $m\ge 1$.}

\marchTen{The first claim imposes constraints on what happens at points ``under" a forward jump such as, in Figure \ref{fg:arcDecomposition}, $\times_2$ under $\alpha^+_0$.} \mayonefour{For an extension $\alpha_i^\varepsilon$ of $\alpha_i$, $a_i^\varepsilon$ denotes the other end of $\alpha_i^\varepsilon$.}

\begin{claim} \label{claim1''}
Suppose that $\alpha_j^+$ is a coherent extension of $\alpha_j$.  If there is an $\ell\ge j+2$ such that $a^+_j$ is in $\alpha_\ell\setminus \{\times_\ell\}$,  
then:
\begin{enumerate}[ref=(\arabic*)]
\item\label{it:minusExtensions} for each \greenchange{$i=j+2,j+3,\dots, \ell$}, \redchange{$\alpha_i^-$ is disjoint from $\alpha_j^+\setminus \{a^+_{j}\}$; and 
\item\label{it:j+1HatPlus} either there exists $i\in\{j+2,j+3,\dots,\ell\}$ such that $a_j^+\in \alpha^-_i$, or $\alpha_{j+1}^+$ intersects $\alpha_j^+$}.
\end{enumerate}

Likewise, suppose $\alpha_j^-$ is a coherent extension of $\alpha_j$.  If there is an $\ell\le j-2$ such that $a^-_j$ is in \greenchange{$\alpha_\ell\setminus \{\times_{\ell+1}\}$},  
then:
\begin{enumerate}[start=3,ref=(\arabic*)]\item for each $i=\ell,\ell+1,\dots,j-2$, \redchange{$\alpha_i^+$ is disjoint from $\alpha_j^-\setminus\{a^-_j\}$; and \item\label{it:backwardChoice} either\marchTwentysix{ there exists} $i\in \{\ell,\ell-1,\dots,j-2\}$\marchTwentysix{ such that $a^-_j\in \alpha^+_i$}, or $\alpha_{j-1}^-$ intersects $\alpha_j^-$}.
\end{enumerate}
\end{claim}

\begin{proof}
We prove the first statement; the ``likewise" is the same, but for the traversal of $A$ in the reverse direction.   

By way of contradiction, suppose first that, for some $i\in\{j+2,j+3,\dots, \ell\}$, $\alpha_i^-$ \redchange{intersects $\alpha_j^+\setminus \{a^+_{j}\}$; let $\meet ij$ be the first intersection with $\alpha^+_j$ as we traverse $\alpha_i^-$ from $\times_i$.   The interior of \mayseven{the subarc $\alpha^-_i[\times_i,\meet ij]$} of $\alpha^-_i$ from $\times_i$ to $\meet ij$ is on the side of the unique simple closed curve contained in $A\cup \alpha^+_j$ that is opposite to the side that contains \mayonefour{$A[s,\times_{j+1}]$}. }  

See Figure \ref{fg:claim1} for an illustration of this proof.

\begin{figure}[ht]
\begin{center}
\includegraphics[scale=1]{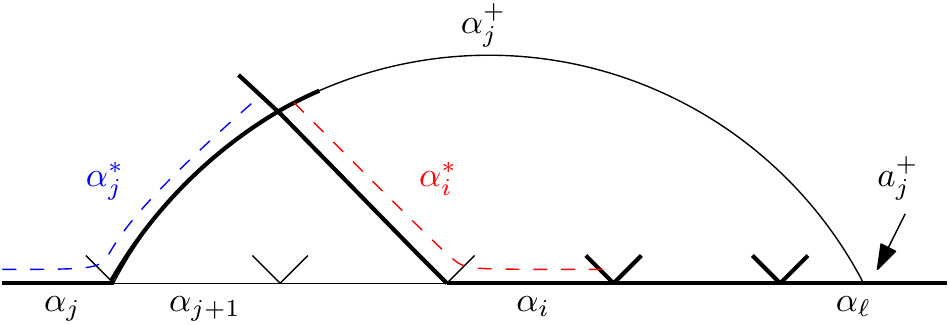}
\caption{Illustration for the proof of Claim 1.
}
\label{fg:claim1}
\end{center}
\end{figure}

\redchange{Let $\alpha^*_j$ be the subarc of $\gamma_j$ consisting of $\alpha_j$ and the portion of $\alpha_j^+$ \redchange{from $\times_{j+1}$ to $\meet ij$}.  Likewise, let $\alpha^*_i$ be the subarc of $\gamma_i$ consisting of $\alpha_i$ and $\alpha_i^-[\times_i,\meet ij]$.  \decTwoOhTwoOh{The arc $A'$ with decomposition \maytwonine{$\alpha_0\dots\alpha_{j-1}\alpha_j^*\alpha_i^*\alpha_{i+1}\dots\alpha_m$}} has smaller weight than $A$.  Also, even if $\meet ij=\times_{j+1}$, $A'$ is a spiral.  (In case $\meet ij=\times_{j+1}$, then $\gamma_j$, $\gamma_{j+1}$, and $\gamma_i$ all cross at $\times_{j+1}$.  This ensures that the cyclic rotation of these three curves at $\times_{j+1}$ is $\alpha_j,\alpha_{j+1},\alpha_i^-,\alpha_j^+,\alpha_{j+1}^-,\beta_i$, where $\beta_i$ is the continuation of $\gamma_i$ from $a_i^-=\meet ij$.)}

\redchange{To see that $A'$ is coherent, first let $C$ be the simple closed curve $(\alpha^*_j\setminus \alpha_j)\cup (\alpha^*_i\setminus \alpha_i)\cup A[\times_{j+1},\times_i]$. For each of the segments $\alpha'$ of $A'$, $\alpha'$ contains some segment $\alpha_k$ of $A$.  Let $\alpha^\varepsilon_k$ be a coherent extension of $\alpha_k$ for $A$.  Follow $\alpha^\varepsilon_k$ from its end in $\alpha_k$ (or, if $k\in \{j,i\}$, from $\meet ij$). If we never encounter $C$, then we arrive at $A'$ on the same side.  On the other hand, if we encounter $C$, it is not at a point in $A$ and so it is in \greenchange{$(\alpha^*_j\setminus \alpha_j)\cup (\alpha^*_i\setminus \alpha_i)$}. }

\greenchange{Label as the {\em outside of $C$\/} the side of $C$ containing \mayonefour{$A[s,\times_{j+1}]$}.  Since $\alpha^\varepsilon_k$ starts on the outside of $C$, its first intersection with $C$ is from that side.  Thus, the portion of $\alpha^\varepsilon_k$ up to that first intersection with $C$ is a coherent extension of $\alpha'$, as required.}

To complete the proof that $A'$ is coherent, we note that, if, for the segment $\alpha$ of $A'$, both $\alpha^-$ and $\alpha^+$ are coherent, then they are contained in coherent extensions of the corresponding segment of $A$.  Since these extensions for $A$ \greenchange{are distinct, as extensions for $A'$ they are also distinct.  Thus, there is a coherent spiral with weight less than $m$, a contradiction.}

\greenchange{For \ref{it:j+1HatPlus}, if, for each $i=\ell,\ell+1,\dots,j-1$, $a_j^+\notin \alpha^-_i$ and $\alpha^+_{j+1}$ is disjoint from $\alpha^+_j$, then $\alpha^+_{j+1}$ and (using \ref{it:minusExtensions}) each $\alpha^-_i$ is a coherent extension in  $A[\times_{j+1},a^+_j]$.   The same argument as in the preceding paragraph shows that $A[\times_{j+1},a^+_j]$ is a coherent spiral with smaller weight than $A$, a contradiction.} 
\end{proof}

The next claim considers ``reverse" coherent extensions.   \marchTen{This claim rules out the possibility of an extension such as $\alpha^-_3$ in Figure \ref{fg:arcDecomposition}.}

\begin{claim} \label{claim2''}  There do not exist $j,\ell\in \{0,1,2,\dots,m\}$ such that $\ell<j$ and $\alpha_j^+$ is a coherent extension with an end in $\alpha_\ell$. 

Likewise,
there do not exist $j,\ell\in \{0,1,2,\dots,m\}$ such that $j<\ell$ and $\alpha_j^-$ is a coherent extension with an end in $\alpha_\ell$.
\end{claim}

\begin{proof}  We \mayseven{only prove the first statement.} 
Choose the least $j$ for which such an $\ell<j$ exists. Suppose first that $a^+_j\in \alpha_{j-1}$.    Thus, $\gamma_{j-1}\cap \gamma_j\supseteq\{a^+_j,\times_j\}$.  \mayseven{ Because $\alpha_j^+\ne \alpha_j^-$,  $a^+_j\ne \times_j$. \mayonefour{Thus,} {Hypothesis} \ref{it:arrpseudo} implies} $\gamma_{j-1}\cap \gamma_j= \{a^+_j,\times_j\}$.
Let $C$ be the simple closed curve \greenchange{contained in} $\alpha_{j-1}\cup \alpha_j\cup\alpha_j^+$.  
Except for $\times_j$, $\alpha_{j-1}^+$ is disjoint from $C$.  

\greenchange{Note that $a^+_{j-1}$ is not in $\alpha_{j-1}\cup \alpha_j$.  Therefore, $a^+_{j-1}$ and the start of $\alpha_{j-1}^+$ from $\times_j$} are on different sides of $C$, which is impossible.  Thus, $a_j^+\notin \alpha_{j-1}$.

\bluechange{The choice of $j$ implies that either $\alpha_{j-1}^+$ is not coherent or it does not intersect $A[a_j^+,\times_j]$.  Therefore, $\alpha_{j-1}^+$ must intersect $\alpha_j\cup \alpha_j^+$ at a point other than $\times_j$.  This gives the two crossings of $\gamma_{j-1}$ with $\gamma_j$.    
}

\bluechange{An intersection of $\alpha_{j-1}^-$ with $\gamma_j$ yields a third intersection of $\gamma_{j-1}$ with $\gamma_j$, and the theorem is proved.  Therefore, we assume $\alpha_{j-1}^-$ is disjoint from $\alpha^-_j$.   On the other hand, the choice of $j$ implies that, for each $k$ with $\times_{k+1}$ in the interior of $A[a^-_j,\times_j]$, $a_k^+\ne a_j^-$.  \decTwoOhTwoOh{Therefore   Claim \ref{claim1''} \ref{it:backwardChoice}} shows $\alpha_{j-1}^-$ is not disjoint from $\alpha^-_j$, the final contradiction.
}  
\end{proof}

We may now suppose that there does not exist an $\alpha_j^+$ that is a coherent extension with an end in any $\alpha_k$ such that $k<j$. Similarly, we may assume that there does not exist an $\alpha_j^-$ that is a coherent extension with any end in any $\alpha_k$ such that  $k>j$.

The final claim combines the first two to completely determine the nature of a coherent extension.  Before we get to it, we require one more detail.

\begin{claim}\label{cl:nonCoherence}
 Let $k\in\{1,2\dots,m\}$.  \redchange{Suppose $\alpha_{k-1}^-$ is not a coherent extension and that $\gamma_k\setminus \alpha_k^-$ has an intersection with $\gamma_{k-1}$.  \febSixteenChange{Then} $\alpha_{k}^-$ is not coherent.}

\redchange{Likewise, if  $\alpha_k^+$ is not coherent and $\gamma_{k-1}\setminus \alpha_{k-1}^+$ has an intersection with $\gamma_k$, then $\alpha_{k-1}^+$ is not coherent.}
\end{claim}

\begin{proof}  We \mayseven{only prove the first statement.} 
Because $A$ is coherent and $\alpha_{k-1}^-$ is not a coherent extension, $\alpha_{k-1}^+$ is a coherent extension of $\alpha_{k-1}$.  \greenchange{Let $\times$ be the intersection of $\gamma_k\setminus \alpha^-_k$ with $\gamma_{k-1}$; it follows that $\gamma_{k-1}\cap \gamma_k\supseteq\{\times_k,\times\}$.  Since $\times\notin\alpha_k^-$,  $\times$ is neither $\times_k$ nor $a^-_k$.  
\mayseven{{Hypothesis} \ref{it:arrpseudo} implies}
$\gamma_{k-1}\cap \gamma_k=\{\times_k,\times\}$.   
}

 \redchange{\greenchange{Since $\alpha_k^-\setminus \{\times_k\}$ is disjoint from $\gamma_{k-1}$;  in particular, it is disjoint from $\alpha_{k-1}$.}   For $k=1$, the preceding claim  implies that  $\alpha_k^-$ is not coherent.  Thus, we suppose $k\ge 2$.}

  The union of $\alpha_{k-1}^-$ and \mayseven{ $A[\times_{k-1},a_{k-1}^-]$}   is a simple closed curve $C$.  Let $p$ be a point of $\alpha_k^-$ near $\times_k$.  From $p$ trace an arc $\delta$ alongside $\alpha_{k-1}$, across $\alpha_{k-1}^-$ and, continuing beside $A$, on to a point near the end \redchange {$\times_0$ of $A$.  Thus, $\delta$ is along the side of $A$ faced by all the $\times _i$.  Because $\alpha_{k-1}^-$ is not coherent, it does not return to $A$ on this side and, therefore, $\delta$ crosses $\alpha_{k-1}^-$ only once.  Consequently, $\delta$ crosses $C$ only once.}   

Suppose by way of contradiction that $\alpha_k^-$ is a coherent extension of $\alpha_{k}$. Because $\alpha^-_k$ does not intersect $\gamma_{k-1}\setminus \{\times_k\}$, it cannot cross $C$.  Therefore, it does not intersect the portion of $\delta$ from its crossing \bluechange{with $\alpha_{k-1}^-$ to its end near $\times_0$.  In particular, } $\alpha_k^-$ has no end in $\alpha_0\alpha_1\cdots\alpha_{k-2}$.  The first paragraph shows $\alpha_k^-$ is also disjoint from $\alpha_{k-1}\setminus\{\times_k\}$.  Thus, \mayonefour{$a_k^-$ is in $A[\times_{k+1},t]$},  contradicting the preceding claim. \end{proof}

We are now ready to give a complete description of the coherent extensions of $A$.

\begin{claim} \label{claim3''}
For $0 \leq j \leq m-1$, if $\alpha_j^+$ is a coherent extension of $\alpha_j$, then $a_j^+\in\alpha_{j+1}$. 

Likewise, for $1 \leq j \leq m$, if $\alpha_j^-$ is  a coherent extension of $\alpha_j$, then $a_j^-\in \alpha_{j-1}$.
\end{claim}
\begin{proof}  We \mayseven{only prove the first statement.} 
 Suppose that $j$ is the least index such that $\alpha_j^+$ is a coherent extension of $\alpha_j$ and $a_j^+\notin\alpha_{j+1}$. 
We show by induction that, for each $k=0,1,\dots,j$,  $\alpha_k^-$ is not a coherent extension of $\alpha_k$. 
For $k=0$, Remark \ref{rk:alpha0-NotCoherent} shows that  $\alpha_0^-$ is not a coherent extension of $\alpha_0$. Now let $k\ge 1$ and suppose that $\alpha_{k-1}^-$ is not a coherent extension of $\alpha_{k-1}$.

Since $A$ is coherent and $\alpha_{k-1}^-$ is not a coherent extension of $\alpha_{k-1}$, we have that  $\alpha_{k-1}^+$ is a coherent extension of $\alpha_{k-1}$.  By the choice of $j$, $a^+_{k-1}\in \alpha_k$.  The preceding claim implies that $\alpha_k^-$ is not a coherent extension of $\alpha_k$\marchTwentysix{, completing the proof that,} for each $k=0,1,\dots,j$, $\alpha_k^-$ is not a coherent extension of $\alpha_k$.

\redchange{Let $\ell$ be such that \greenchange{$a_j^+\in\alpha_\ell\setminus\{\times_\ell\}$}.     The second claim and the choice of $j$ \febSixteenChange{show that  $\ell>j+1$}.  Part \ref{it:minusExtensions} of the first claim implies that, for each $i=j+2,j+3,\dots,\ell$, $\alpha_i^-$ is disjoint from $\alpha_j^+\setminus\{a^+_j\}$, while the second claim asserts that $a^+_j\notin \alpha^-_i$.  Part \ref{it:j+1HatPlus} of the first claim now implies that $\alpha^+_{j+1}$ intersects $\alpha^+_j$ at a point $q$.}

We showed that $\alpha_j^-$ is not a coherent extension of $\alpha_j$ and that $\alpha^+_{j+1}$ intersects $\gamma_j$ at $\times_{j+1}$ and $q$.  Consequently, \marchTwentysix{the coherence of $A$ and} Claim \ref{cl:nonCoherence} show that $\alpha_{j+1}^+$ is a coherent extension.  The second claim shows that, for some $r>j+1$, \greenchange{$a_{j+1}^+\in\alpha_r\setminus \{\times_r\}$}.   The first two claims  show that, for each \greenchange{$i=j+3,j+4,\dots,r$, $\alpha^-_i$} is disjoint from $\alpha^+_{j+1}$.

Let $\alpha_{\ell}^*$ be the subarc $\alpha_\ell[\times_\ell,a^+_j]$ and let $A'$ be the arc consisting of $\alpha_{j+1}$, $\alpha_{j+2},\dots,$ $\alpha_{\ell-1},\alpha_\ell^*$.   Just above, we showed that $\alpha_{j+2}^-$ is disjoint from $\alpha_j^+$.  Therefore, $\alpha^-_{j+2}$ is a coherent extension of $\alpha_{j+2}$ with respect to  $A'$. The second claim shows that $\alpha_{j+2}^-$ has both ends in $\alpha_{j+1}$. 

Since $\alpha_{j+2}^-$ intersects $\alpha_{j+1}$ in $\times_{j+2}$ and $a_{j+2}^-$, \febSixteenChange{we see that $a^+_{j+1}\notin \alpha_{j+2}$; therefore} $r>j+2$. Moreover, it follows that $\alpha_{j+2}^+$ is disjoint from $\alpha_{j+1}^+$. This, together with the fact that, for each \greenchange{$i=j+3,j+4,\dots,r$}, $\alpha^-_i$ is disjoint from $\alpha^+_{j+1}$, contradicts the first claim. 
\end{proof}

\marchTwentysix{Because $A$ is coherent and Remark \ref{rk:alpha0-NotCoherent} shows $\alpha_0^-$ is not a coherent extension of $\alpha_0$,}  $\alpha_0^+$ is a coherent extension of $\alpha_0$\marchTwentysix{.  Likewise,} $\alpha_m^-$ is a coherent extension of $\alpha_m$.  It follows that there is a $j\ge 1$ such that $\alpha_{j-1}^+$ is a coherent extension of $\alpha_{j-1}$ and $\alpha_j^-$ is a coherent extension of $\alpha_j$.  The fourth claim  implies that both ends of $\alpha_{j-1}^+$ are in $\alpha_j$ and both ends of $\alpha_j^-$ are in $\alpha_{j-1}$.  This implies that $|\gamma_{j-1}\cap \gamma_j|\ge 3$, \mayseven{completing the proof that \ref{it:arrpseudo}${}\Rightarrow{}$\ref{it:noCoherSpiral}}.

\mayseven{\decTwoOhTwoOh{\ref{it:noCoherSpiral}}${}\Rightarrow{}$\ref{it:SpiralsAllExternal}.  For this argument, we will make use of the following trivial \decTwoOhTwoOh{observation}.}  

\begin{observation}\label{obs:trivial}
Let $Q,R,S$ be arcs in the sphere such that $R$ and $S$ both have their ends in $Q$, but otherwise \mayoneseven{are disjoint from $Q$}.  We assume $R$ and $S$ have finitely many intersections and these are all crossings.

\mayonefour{Assume that short subarcs of $S$ starting at each end of $S$} are on different sides of the unique simple closed curve in $R\cup Q$.  Then $(R\cap S)\setminus Q$ has at least one point.  
\end{observation}

\mayseven{
\octthirteen{By way of contradiction,} 
let $A$ be {a least-weight spiral in $P(\Gamma)$} having no segment external for $A$.}   Since the only segment in an arc of weight 0 is external for that arc, $A$ has positive weight $m$.  Let $\alpha_0\alpha_1\cdots\alpha_m$ be the decomposition of $A$.

\decTwoOhTwoOh{By assumption,} $A$ is incoherent; let $\alpha_i$ be an incoherent segment of $A$.  
By definition, both $\alpha_i^-$ and $\alpha^+_i$ are both incoherent extensions.  Thus, \mayonefour{short subarcs near their ends $a_i^-$ and $a_i^+$} are on the side of $A$ that is opposite the side faced by all the crossings $\times_1,\dots,\times_m$.  We remark that one or both of $a_i^-$ and $a_i^+$ might be in $\{\times_0,\times_{m+1}\}$.  If $i=0$, then only \mayonefour{the subarc near} $a_0^+$ is forced by incoherence to be on the side of $A$ not faced by the crossings; the fact that $\gamma_0$ is a simple closed curve implies $a_0^-$ is as well.  An analogous statement applies if $i=m$.

As we traverse $A$ from $\times_0$ to $\times_{m+1}$, we first encounter $a_i^-$ and then $a_i^+$ (these could be equal).   Therefore, either $a_i^-\in \alpha_0\alpha_1\cdots\alpha_{i-1}$ or $a_i^+\in \alpha_{i+1}\alpha_{i+2}\cdots\alpha_m$.  As these are symmetric up to reversal of $A$, we assume the latter.

Let $B$ be the spiral $\alpha_{i+1}\alpha_{i+2}\cdots\alpha_m$; evidently, its weight is less than that of $A$.  \mayseven{Therefore, $B$ has a segment $\alpha_j$ that is external for $B$}.
\maysix{Notice that the side of $B$ faced by all its crossings is separated from $A[\times_0,\times_{i+1}]\setminus\{\times_{i+1}\}$ by the simple closed curve $A[\times_{i+1},a_i^+]\cup \alpha_i^+$.}

Since \mayseven{no segment of $A$ is external for $A$}, $\gamma_j$ intersects $A$ at a point outside~$\alpha_j$.  This implies that $\alpha_j^-$ and $\alpha_j^+$ intersect $A$ at points $a_j^-$ and $a_j^+$, respectively, that are not in $\alpha_j$.    Since $\alpha_j$ is external for $B$, $a_j^-$ and $a_j^+$ are in~$A\setminus B$.  

\begin{claim}  Both $\alpha_j^-$ and $\alpha_j^+$ have an intersection with $\alpha_i^+$, and these intersections are distinct points of~$\alpha_i^+$.
\end{claim}

\begin{proof}
In an extreme case, $j=i+1$; here $\times_{i+1}$ is the required common point between $\alpha_j^-$ and $\alpha_i^+$.  In the other extreme case,  $a_i^+ = \times_{m+1}$ and $j=m$; here $\times_{m+1}$ is the common point between $\alpha_j^+$ and $\alpha_i^+$.  (These could happen simultaneously.)
Otherwise, $\times_j$ and $\times_{j+1}$ are interior points of $B$.

If we take $Q=A$, $R=\alpha_i^+$ and {$S$ to be either $\alpha_j^-$ or $\alpha_j^+$}, then {Observation \ref{obs:trivial}} {implies there are distinct points, one in each of $(\alpha_i^+\cap \alpha_j^-)\setminus A$ and $(\alpha_i^+\cap \alpha_j^+)\setminus A$.} 
\end{proof}

\mayoneone{Since $|\gamma_i\cap \gamma_j|\le 2$, the only points} in $\gamma_i\cap \gamma_j$ are the points in each of $\alpha_j^-\cap\alpha_i^+$ and $\alpha_j^+\cap \alpha_i^+$.   In particular, $\alpha_j^-$ and $\alpha_j^+$ are both disjoint \mayonefour{from $\alpha_i^-\setminus \{a_i^-\}$.  
}  

\mayonefour{Because $a_j^-\notin B$ and $\alpha_j^-$ is disjoint from $\alpha_i^-\setminus\{a_i^-\}$,  $a_j^-\in A[\times_0,\times_i]$.  On the other hand,  the disjointness of $\alpha_j^-$ with $\alpha_i^-\setminus\{a_i^-\}$ implies that $a_i^-$ is not further from $\times_0$ in $A$ than $a^-_j$ is.  In turn, this implies that $a_i^-$ is in $A[\times_0,\times_i]\setminus\{\times_i\}$.  Therefore, reversing the direction of traversal of $A$,  $\alpha_i^-$ \maytwonine{may play the role of $\alpha_i^+$ in the preceding argument}. } 

Thus, there is a $k\in \{0,1,\dots,i-1\}$ such that $\alpha_k$ is a \mayseven{segment of $\alpha_0\alpha_1\cdots\alpha_{i-1}$, external for $\alpha_0\alpha_1\cdots\alpha_{i-1}$, but not external for $A$.}  The argument above for $\alpha_i^+$ and $\alpha_j$ applies to $\alpha_i^-$ and $\alpha_k$.  Thus, $\alpha_k^-$ and $\alpha_k^+$ have their endpoints \maysix{$a_k^-$ and $a_k^+$, respectively, \mayoneone{in $A[\times_{i+1},\times_{m+1}]\setminus\{\times_{i+1}\}$}}.  \decTwoOhTwoOh{Observation \ref{obs:trivial}} implies that each of $\alpha_j^-$ and $\alpha_j^+$ has an intersection with each of $\alpha_k^-$ and $\alpha_k^+$.  In particular, $\gamma_j\cap \gamma_k$ has at least four points.  \decTwoOhTwoOh{This contradicts \ref{it:arrpseudo}, which is equivalent to \ref{it:noCoherSpiral}}.
\end{cproofof}

\section{A pseudocircle transversal}\label{sec:transversal}

\janone{We recall that the proof of Theorem \ref{th:main} has two parts.  For our current set $\Gamma$ of pairwise intersecting pseudocircles, we must (i) find \janseven{a pair of initial approximations} to the pseudocircle for the next edge and (ii) \janseven{show that the pair of approximations imply the existence of \janthirteen{the desired curve crossing all the curves in $\Gamma$}}.  In this section, our focus is on the second of these \decTwoOhTwoOh{parts}.}

\janone{In particular, in Section \ref{sec:hCxHasExactExtension}, we will show \janseven{how to find two initial approximations that together} intersect all the curves in our current $\Gamma$.  The main result of this section is to use these two approximations to find the single curve that intersects all the curves in $\Gamma$.  }

\janthirteen{We are reminded of the theorems that are: 
\begin{itemize}[topsep=0pt, itemsep=0pt,leftmargin=11pt]\item Helly-type:  if a collection of sets is such that every $k$ of the sets admits a transversal, then the whole collection admits a transversal; and
\item Gallai-type: if a collection of sets is such that every $k$ of the sets admits a transversal, then the whole admits a small set of partial transversals whose union is a transversal. 
\end{itemize} 
Our theorem has the following different character:
if a collection of sets admits a small set of partial transversals whose union is a transversal, then it admits a transversal.  
We do not know \decTwoOhTwoOh{of another} example of this type of theorem.}

\begin{definition}
\janone{Let $\Gamma$ be an arrangement of pseudocircles.  
\begin{itemize} \item A set $\Lambda$ of simple closed curves is a {\em $\Gamma$-transversal\/} if every curve in $\Gamma$ intersects at least one of the curves in $\Lambda$. \item A simple closed curve $\gamma$ is a {\em $\Gamma$-pseudocircle\/} if $\Gamma\cup \{\gamma\}$ is an arrangement of pseudocircles.   \end{itemize}}
\end{definition}

\begin{theorem}
\label{thm:sweeping}
\janone{Let $\Gamma$ be an arrangement of pseudocircles.
Let $\gamma_1$ and $\gamma_2$ $\Gamma$-pseudocircles  such that $\{\gamma_1,\gamma_2\}$ is a $\Gamma$-transversal.
Suppose \begin{enumerate}[label={\bf(\ref{thm:sweeping}.\arabic*)},leftmargin=63pt,ref=(\ref{thm:sweeping}.\arabic*)]
\item $\gamma_1\cap \gamma_2$ is a non-trivial arc and 
\item\label{it:deltaIntersect} if $\delta_1,\delta_2\in \Gamma$ are such that $\delta_1\cap \gamma_1=\varnothing$ and $\delta_2\cap \gamma_2=\varnothing$, then $\delta_1\cap \delta_2\ne\varnothing$.
\end{enumerate}
Then there exists a $\Gamma$-pseudocircle $\gamma$ containing $\gamma_1\cap \gamma_2$ and \janthirteen{$\gamma\setminus(\gamma_1\cap \gamma_2)$} is contained in the closure of the face $F$ of $\gamma_1\cup\gamma_2$ not incident with \janseven{$\gamma_1\cap \gamma_2$}.}
\end{theorem}

\janone{Our proof shows that one can sweep \janseven{either of $\gamma_1$ or $\gamma_2$} to the required $\gamma$.  In their classic paper \cite{sweeping}, Snoeyink and Hershberger show how to sweep one curve through the others in an arrangement of pseudoarcs and pseudocircles.  In particular, up to sweeping, $\gamma$ is unique.}

\janseven{\decTwoOhTwoOh{In the proof of Theorem \ref{th:main}, any two curves in the current $\Gamma$ intersect}, so Hypothesis \ref{it:deltaIntersect} holds automatically.}

\bigskip
\begin{cproofof}{Theorem \ref{thm:sweeping}}\octeight{\janone{Let $\sigma=\gamma_1\cap \gamma_2$ and, for} $i=1,2$, let \[\Gamma_i=\{\delta\in \Gamma\;:\; \delta\cap \gamma_i\neq \emptyset \}\,.\] \janone{Because $\{\gamma_1,\gamma_2\}$ is a $\Gamma$-transversal,}  $\Gamma_1\cup \Gamma_2=\Gamma$. Let \janone{$n= n(\Gamma, \gamma_1,\gamma_2)=
|\Gamma_2\setminus  \Gamma_1|$}. Define $k=k(\Gamma,\gamma_1, \gamma_2)$ as the number of crossings in $P(\Gamma)$
included in the face  $F$ of $\gamma_1\cup \gamma_2$ not incident with $\sigma$.   We proceed by induction on $n+k$.
}

\octeight{
We can assume that neither $\Gamma_1\setminus \Gamma_2$ nor $\Gamma_2\setminus \Gamma_1$ is empty, else we pick $\gamma$ to be either equal to $\gamma_2$ or $\gamma_1$. }


\octeight{
If there is an arc $\alpha$ of some $\delta\in \Gamma_2\setminus \Gamma_1$ \janthirteen{incident with a face of $P(\Gamma\cup \{\gamma_1,\gamma_2\})$ that is included in $F$ and incident with $\gamma_1$}, then by shifting some part of \janone{$\gamma_1$} to cross $\alpha$ via a Reidemeister Type II move, we obtain a curve $\gamma_1'$ such that the pair $(\gamma_1',\gamma_2)$ satisfies the same hypothesis as $(\gamma_1, \gamma_2)$. Since  \janone{$n(\Gamma,\gamma_1',\gamma_2)+k(\Gamma,\gamma_1',\gamma_2)<n(\Gamma,\gamma_1',\gamma_2)+k(\Gamma,\gamma_1',\gamma_2)$}, the result follows by induction. 
}

\octeight{
In the alternative, there exists an arc $A$ with ends in $\gamma_2\setminus \sigma$, but otherwise contained in $F\cap P(\Gamma_1\cup\{\gamma_1\})$, such that $\gamma_2\cup A$ separates   $\gamma_1\setminus \sigma$ from $P(\Gamma_2\setminus \Gamma_1)\cap F$. 
Let \janThree{$\Delta_{A}$} be the  \marchTen{closure of the component of $F\setminus A$ that is} incident with both $A$ and $\gamma_1 \setminus \sigma$. }
 Among the finitely many choices for $A$, we choose $A$ so that \janThree{$\Delta_{A}$ is minimal under inclusion}.  \maythreeone{See Figure \ref{fg:DeltaA}.} 
 
 \begin{figure}[ht]
\begin{center}
\includegraphics[scale=0.5]{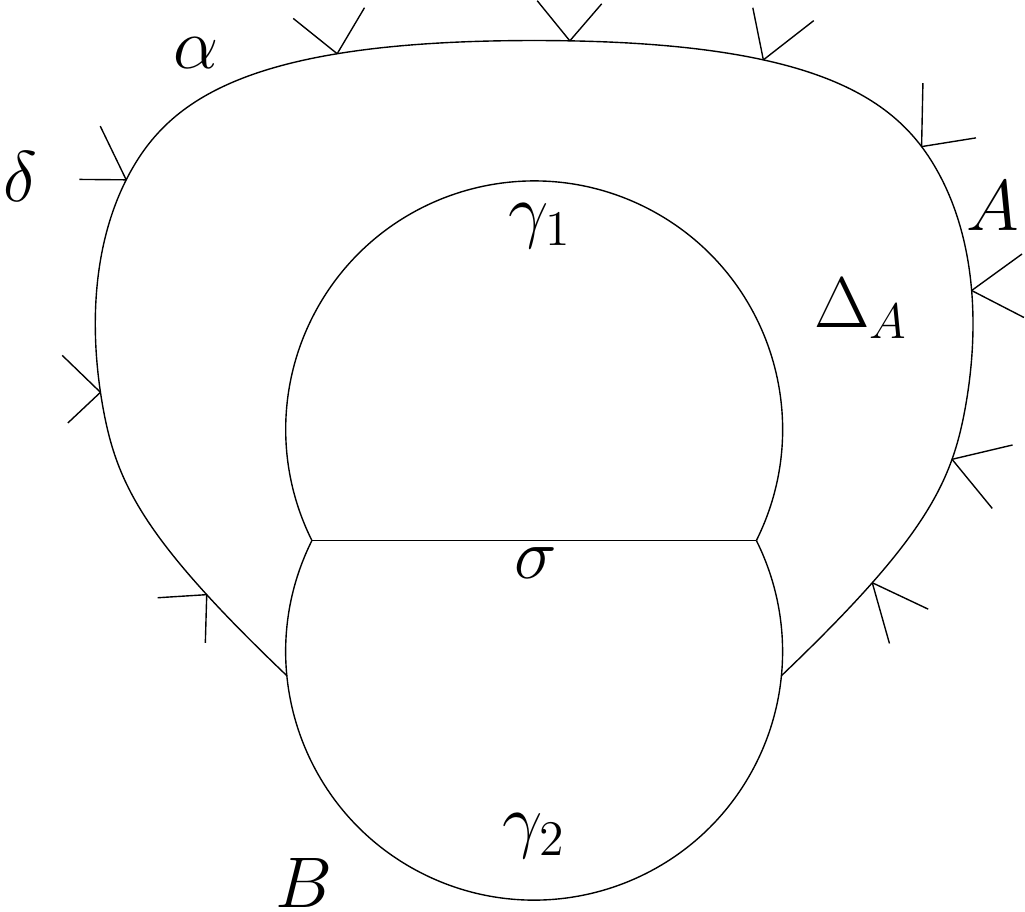}
\end{center}
\caption{The region $\Delta_A$.}  \label{fg:DeltaA}
\end{figure}

\janone{We apply Theorem \ref{th:technical}, \janseven{specifically} \ref{it:arrpseudo}${}\Rightarrow{}$\ref{it:noCoherSpiral}, to see that there is a crossing in $A$ facing $\Delta_A$.  The proof is by contradiction, assuming that every crossing faces the \octeight{other side} of $A$.  In particular, $A$ is a spiral; the contradiction arises from the following.}

\maytwonine{\begin{claim}\label{cl:crossToDeltaA}  \janone{There is a crossing in $A$ facing $\Delta_A$.}
\end{claim}}

\begin{proof} \janone{By way of contradiction, suppose every crossing faces the other side of $A$; in particular, $A$ is a spiral.  We show that $A$ is coherent, contradicting Theorem \ref{th:technical} (\janseven{specifically}, \ref{it:arrpseudo}${}\Rightarrow{}$\ref{it:noCoherSpiral}). It suffices to show that if $\alpha$ is an arc in the decomposition of $A$, then $\alpha$ is coherent. Let $\delta$ be the curve in $\Gamma_1$ containing $\alpha$.  }

Let $B$ be the \janthirteen{closed} arc in $\gamma_2\setminus \sigma$ such that $A\cup B$ is a simple closed curve.
Consider the continuation of $\delta$ from one end of $\alpha$.  Since $\delta$ crosses $\gamma_1$ twice, the continuation must eventually reach $\gamma_1$; in particular, it must have a first intersection with \janone{$A\cup B$}. 

We show below that it is impossible for both continuations to have these first intersections \augfifteen{in $B$}.  Therefore, for one of them, the first intersection is in the interior of $A$, showing that $\alpha$ is coherent.


\octeight{
So suppose both continuations intersect $B$ for the first time at $p_1$ and $p_2$. Since $|\gamma_2\cap \delta|\leq 2$, $\gamma_2\cap \delta=\{p_1,p_2\}$. Therefore,  $\delta$ is contained in the region bounded by $A\cup (\gamma_2\setminus B)$ and intersects the boundary of this region only at $\alpha$. This shows that $\delta\notin \Gamma_1$, the required contradiction.}
\end{proof}

\maytwonine{\begin{claim}\label{cl:fromAtoDeltak}  Let \octeight{$\delta\in \Gamma$}. Then  every arc in \octeight{$\delta\cap \Delta_{A}$} has one end in the interior of \octeight{$\gamma_1\setminus \sigma$} and one end not in the interior of \octeight{$\gamma_1\setminus \sigma$}.
\end{claim} }

\begin{proof} \maytwonine{Let $\beta$ be an arc \janone{$\delta\cap \Delta_{A}$}.  If $\beta$ has no end in \janone{$\gamma_1\setminus \sigma$}, then $\beta\cup A$ contains an \febFiveChange{arc $A'$} that separates \janone{$\gamma_1\setminus \sigma$}  from \octeight{$P(\Gamma_2)\cap F$} such that \janThree{$\Delta_{A'}$} is properly contained in \janThree{$\Delta_{A}$}, contradicting the choice of $A$.}

If $\beta$ has no end in the complement of the interior of \octeight{$\gamma_1\setminus \sigma$} in the boundary of \janThree{$\Delta_{A}$}, then $\beta$ is contained in \janThree{$\Delta_{A}$} and has both ends in the interior of \octeight{$\gamma_1\setminus \sigma$}.  \augusttwenty{Since \janone{$\delta\cap \gamma_1$}} has exactly two points, these are the two ends of $\beta$.  The preceding paragraph shows that $\beta$ is the only arc \augusttwenty{in \janone{$\delta\cap \Delta_{A}$}}, \octeight{and hence \janone{$\delta\setminus \beta$} is included in the side of  $\gamma_1$ disjoint from $\gamma_2\setminus \sigma$. Thus, for any \janone{$\delta'\in \Gamma_2\setminus \Gamma_1$, since $\delta'\cap \Delta_A$ and $\delta'\cap \gamma_1$ are empty, this implies that $\delta\cap \delta'=\emptyset$, contradicting \ref{it:deltaIntersect}}.}
\end{proof}

\junesix{\janone{Claim \ref{cl:crossToDeltaA} implies there are  distinct elements \octeight{$\delta$, $\delta'$} of $\Gamma_1$  that have a} crossing \octeight{$\times$} in $A$  \junesix{through which they proceed into} the \janThree{$\Delta_{A}$}-side of $A$.    Claim \ref{cl:fromAtoDeltak} shows \marchTen{both extensions from this crossing are arcs \octeight{$\rho$} and \octeight{$\rho'$} in \augusttwenty{\octeight{$\delta$ and $\delta'$}}, respectively, joining \octeight{$\times$} to their ends \octeight{$a$ and $a'$}, respectively, in the interior of \octeight{$\gamma_1\setminus \sigma$}.  All intersections of \augusttwenty{$\delta$ and $\delta'$} with \octeight{$\gamma_1$} are crossings, so this is true in particular \decTwoOhTwoOh{for \octeight{$a$ and $a'$}.}}}

\maytwonine{Since \octeight{$\rho$ and $\rho'$} cross at \octeight{$\times$}, they have at most one other crossing.  Let \octeight{$\times^*$} be that other crossing if it exists; otherwise \octeight{$\times^*$ is $\times$}.  The union of the subarcs of \octeight{$\rho$} from \octeight{$a$} to \octeight{$\times^*$, $\rho'$} from \octeight{$a'$} to \octeight{$\times^*$}, and \octeight{$\gamma_1\setminus \sigma$} from \octeight{$a$ to $a'$} is a simple closed curve $\lambda$ in the closed disc \janThree{$\Delta_{A}$}. } 

We aim to show that the interiors of the arcs \octeight{$\rho\cap \lambda$} and \octeight{$\rho'\cap \lambda$} are not crossed by any curve in $\Gamma$.  \janone{If this is not the case, then there is a $\mu\in\Gamma$ that crosses, say, $\rho\cap \lambda$.  By Claim \ref{cl:fromAtoDeltak}, the component of $\mu\cap \Delta_A$ containing this crossing has a subarc $\mu'$ with one end in $\rho$ and one end in the interior of $A$.  We may assume $\mu'$ has no other intersection with $\rho$.}

\maytwonine{There is an arc $A'\ne A$ in \octeight{$A\cup \mu'\cup \rho_e$} that separates the interior of \octeight{$\gamma_1\setminus \sigma$} from \octeight{$P(\Gamma_2)\cap F$}.  However, \janThree{$\Delta_{A'}$} is a proper subset of \janThree{$\Delta_{A}$}.  This contradiction shows that no curve in \octeight{$\Gamma$ intersects either $\rho\cap \lambda$ or $\rho'\cap \lambda$}.} 

\maytwonine{It follows that we can perform the equivalent of a Reidemeister III move to shift the portion of $\gamma_1$ between $a$ and \janseven{$a'$} across $\times^*$.  \janseven{(It is possible that $a=a'$.  In this case, we still do the Reidemeister III move, starting by shifting $\gamma_1$ into the face bounded by $\lambda$.  In the context of Theorem \ref{th:main}, this case does not occur.)}  If $\gamma_1'$ is the resulting curve, then \janone{$n'(\Gamma,\gamma_1',\gamma_2)=n$, $k'(\Gamma,\gamma_1',\gamma_2)=k-1$}, and hence the result follows from \deceight{the induction}.}
\end{cproofof}

\section{h-Convex and pseudospherical drawings }\label{sec:hcxPs}

In this section we introduce h-convex drawings and  prove the easy implication \ref{it:Dwps}${}\Rightarrow{}$\ref{it:Dhc}. 

The following notions were introduced by Arroyo et al.~\cite{convex}.  \decTwoOhTwoOh{(Here, for a drawing $D$ of a graph $G$, if $T$ is a subgraph of $G$, then $D[T]$ denotes the subdrawing of $D$ corresponding to $T$.  If $T$ is the subgraph idnduced by an edge $e$ or if $T$ has vertex set $S$ and no edges,  then we write simply $D[e]$ or $D[S]$, respectively.)}

\begin{definition}\label{df:convex}
Let $D$ be a drawing in the sphere of the complete graph $K_n$ \mayonefour{in which any two edges have at} most one point in common \janThree{and that this point, if it exists, is} either a common incident vertex or a crossing point.
\begin{enumerate} [label = {\bf(\ref{df:convex}.\arabic*)},ref=(\ref{df:convex}.\arabic*), leftmargin = 60pt]
\item\label{it:convexSide} Let $T$ be a 3-cycle in $K_n$.  A closed disc $\Delta$ bounded by $D[T]$ is {\em a convex side of $D[T]$\/} if, for every two vertices $x,y$ such that $D[\{x,y\}]\subseteq \Delta$, then $D[xy]\subseteq \Delta$.
\item The drawing $D$ is {\em convex\/} if, for every 3-cycle $T$ in $K_n$, $D[T]$ has a convex side.
\item\label{it:defHconvex} If $D$ is convex, then $D$ is {\em h-convex\/} (short for ``hereditarily convex") if there exists a set $\mathcal C$ consisting of a convex side $\Delta_T$ for every 3-cycle $T$ of $K_n$,  such that, for any two 3-cycles $T,T'$, if $D[T]\subseteq \Delta_{T'}$, then $\Delta_T\subseteq \Delta_{T'}$.
\item The drawing $D$ is {\em f-convex\/} (short for ``face convex") if there is a face $F$ of $D[K_n]$ such that, for every 3-cycle $T$ of $K_n$, the closed disc $\Delta_T$ bounded by $D[T]$ and disjoint from $F$ is convex.
\end{enumerate}
\end{definition}

\octthirteen{A pseudolinear drawing of $K_n$ in the plane is evidently homeomorphic to an f-convex drawing in the sphere.  The converse is proved in \cite{amrs}: taking a witnessing face of an f-convex drawing of $K_n$ in the sphere to be the outer face yields a pseudolinear drawing of $K_n$ in the plane}.  

\octthirteen{Arrangements of pseudolines naturally correspond to rank 3 oriented matroids \cite[Def.\ 5.3.1]{oriented}.  Theorem \ref{th:main} shows that an h-convex drawing of $K_n$ is also equivalent by Reidemeister III moves to a rank 3 oriented matroid.
}

Clearly f-convex drawings are h-convex and h-convex drawings are convex.  
\mayseven{Evidence is given in \cite{convex} to support the conjecture that every \maytwonine{crossing-minimal} drawing of $K_n$ is convex.}  \mayonefour{A polynomial-time algorithm recognizing h-convexity follows from their result that} a drawing of $K_n$ is h-convex if and only if it does not contain as a subdrawing any of the three drawings \mayoneone{(two of $K_5$ and one of $K_6$)} shown in Figure \ref{fg:K5sK6}.  \mayseven{We do not need this result here and there is little overlap of this work with \cite{convex}.} 

\ignore{\decTwoOhTwoOh{A 2-page drawing of $K_n$ is convex.  To see this, let $T$ be a 3-cycle in a 2-page drawing $D$ of $K_n$.    At least two edges of $T$ are on the same closed side of the spine.  If all three on one closed side of the spine, then we take the bounded side $T$ as its convex side.  Since there are no vertices in the interior of $T$, this side is obviously a convex side.  Otherwise, one closed side has at precisely two edges of $T$.  If the vertex incident with these two edges is, in the spine ordering, in the middle of the three vertices of $T$, then the infinite side of $T$ is declared to be the convex side; otherwise the bounded side is the convex side.}

\decTwoOhTwoOh{In the case precisely two edges of $T$ are on one side of the spine, it is not obvious that the chosen side is indeed convex.  Let $u,v,w$ be the vertices of $T$ with $uv$ and $vw$ on the same side of the spine.  If the order on the spine is $u,v,w$, then the chosen side is the infinite side of $T$.  For two vertices $x,y$ on the closed infinite side of $T$, the order on the spine is $x,y,u,v,w$ or $x,u,v,w,y$ and so $xy$ does not cross $T$, as required.  If the order on the spine is $v,u,w$, then any two vertices $x,y$ in the bounded side of $T$ produce the order $v,u,x,y,w$ and again $xy$ does not cross $T$.  Therefore, the drawing $D$ is convex.
}
}

\rbr{\'Abrego et al \cite[Sec.~4.3, 4.4]{abrego2} show that every   crossing-minimal 2-page drawing $D$ of $K_n$ has, up to symmetry, a certain matrix representation of $D$.  It is assumed that the {\em spine\/} of the drawing is the $x$-axis and that each edge is drawn either in the {closed upper} half plane $H^+$ having $y\ge0$ or in the {closed lower} half plane $H^-$ having ${y\le0}$.  For $n$ even, they show there is a unique such drawing, while for $n$ odd, there are some options.}

\rbr{For $n$ even, it is quite straightforward to verify convexity and h-convexity of $D$ by choosing a particular side of each 3-cycle $xyz$.  If $xy,xz,yz$ are either all in $H^+$ or all in $H^-$, then the convex side of $xyz$ is  the bounded side of $xyz$.  If $xy$ and $yz$ are in, say, $H^+$ but $xz$ is not, then there are two cases.  If $y$---the vertex incident with both $H^+$-edges---is in the middle of $x,y,z$ in the linear order on $S$, then the unbounded side of $xyz$ is the convex side.  Otherwise, $y$ is an end of the ordering of $x,y,z$ on $S$ and the bounded side is the convex side.
}

\rbr{The verification of convexity is very simple; h-convexity requires some additional thinking.  Since it is outside the scope of this work, we omit these discussions.  Moreover, for $n$ odd, it is not at all obvious that our arguments hold.  In the case $n$ is odd, the locations of some edges are not determined in \cite{abrego2}.  We expect to be able to provide a detailed analysis of this case in a future publication, which would also include the analysis above for $n$ even.  
}

\begin{figure}[ht!]
\begin{center}
\includegraphics[scale=0.4]{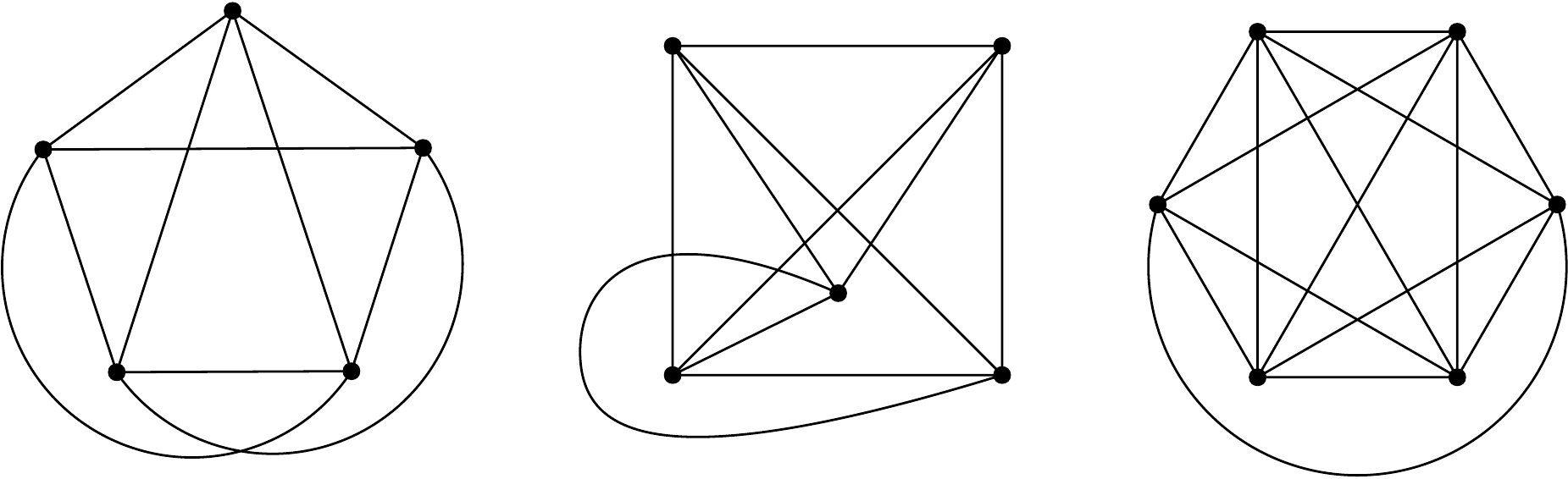}\qquad
\caption{The three obstructions for h-convexity.}\label{fg:K5sK6}
\end{center}
\end{figure}


Our principal goal is to prove Theorem \ref{th:main}.
%
%
As mentioned in the introduction, the implication \ref{it:Dps}${}\Rightarrow{}$\ref{it:Dwps} is trivial.  \decTwoOhTwoOh{To see the implication \ref{it:Dwps}${}\Rightarrow{}$\ref{it:Dhc}, let $e$, $f$, and $g$ be the three edges of a 3-cycle $T$, contained in pseudocircles $\gamma_e$, $\gamma_f$, and $\gamma_g$, respectively.  For each $x\in \{e,f,g\}$, \ref{it:psEdge} shows that the open arc \rbr{$\gamma_x\setminus D[x]$} is contained in one of the two open sides of $D[T]$. Since the $\gamma_x$ cross at their common vertices, it follows that one open side of $D[T]$ is disjoint from all of  $\gamma_e$, $\gamma_f$, and $\gamma_g$; \rbr{we set $\Delta_T$ to be the closure of this side of $D[T]$}.}

\decTwoOhTwoOh{\rbr{To see that $\Delta_T$ is convex, let $u,v$ be vertices in $\Delta_T$}.  Then $u$ and $v$ are on the same side of each of  $\gamma_e$, $\gamma_f$, and $\gamma_g$ and so \ref{it:psEdge} shows $uv$ does not cross any of $\gamma_e$, $\gamma_f$, and $\gamma_g$.  In particular, $uv$ does not cross $D[T]$, so \rbr{$\Delta_T$ is} indeed convex.}

\decTwoOhTwoOh{\rbr{For h-convexity, suppose a second 3-cycle $uvw$ is drawn in $\Delta_T$.  Then} \ref{it:psCrossings} shows that, for example, $\gamma_{uv}$ must cross each of $\gamma_e$, $\gamma_f$, and $\gamma_g$.  Therefore, the side of \rbr{$D[uvw]$ contained in $\Delta_T$ is $\Delta_{uvw}$}, as required.}

%
%
%

The proof of the remaining implication in Theorem \ref{th:main} is given in Section \ref{sec:hCxHasExactExtension}.   
Section \ref{sec:HcxIsPS} provides the necessary discussion of h-convex drawings required to obtain an appropriate initial approximation of the simple closed curve for the ``next" edge of $K_n$ to extend a (partial) collection of curves satisfying \ref{it:psVertex}--\ref{it:psEdge}.

\section{\mayfour{h-convex drawings}}\label{sec:HcxIsPS}

\mayfour{Our main goal now is to prove \octthirteen{the remaining part of} \junesix{Theorem \ref{th:main}:} an h-convex drawing of $K_n$ extends to simple closed curves that satisfy \ref{it:psVertex}, \ref{it:psCrossings}, and  \ref{it:psEdge}.}  
The proof, \junesix{given in the next section, requires three facts about h-convex drawings of $K_n$:  Lemmas \ref{lm:disjtDelta}, \ref{lm:crossingFi}, and \ref{lm:noInterlace} below.  The latter two are straightforward consequences of the first.  However, the proof of the first} is elucidated in this section through a fairly long series of lemmas.  

\junesix{The reader may skip the proof of Lemma \ref{lm:disjtDelta} in order to proceed directly to the proof of Theorem \ref{th:main}.}

\begin{notation}\label{nt:sigmaiAndDi} Let $D$ be an h-convex drawing of $K_n$ and let $\mathcal C$ be a particular choice of convex sides of the 3-cycles witnessing the h-convexity of $D$ as in Definition  \ref{it:defHconvex}.
Let $e$ be any edge of $K_n$ with an arbitrary orientation from one \julythirteen{end of $e$ to the other}.  
\begin{enumerate}[label=({HC\arabic*}), leftmargin = .70truein, start=1, ref=({HC\arabic*})]
\item Set $\Sigma_e^1$ to be the set of all vertices $v$ of $K_n$ not incident with $e$ such that the side in $\mathcal C$ of the 3-cycle containing $v$ and $e$ is the left side, relative to the given orientation of $e$.  \janThree{The remaining vertices not incident with $e$} have their convex side \febSixteenChange{that is in $\mathcal C$} relative to $e$ on the right and they make up $\Sigma_e^2$.  
\item For $i=1,2$, we set $D_e^i$ to be the subdrawing of $D$ \maytwonine{of the complete subgraph} induced by $\Sigma_e^i$ and \janFourteen{the ends of} $e$.
\end{enumerate}
\end{notation}

\maytwonine{Our next lemma is the main point of this section.  Its proof follows its  two  straightforward consequences.}

\begin{lemma}\label{lm:disjtDelta}  Let $D$ be an h-convex drawing of $K_n$ with \febSixteenChange{\deceight{a specified} witnessing set
~of} convex sides.  For $i=1,2$, if $\Sigma_e^i$ is not empty, then there is a closed disc $\Delta_e^i$ that contains $D_e^i$ and  is bounded by a cycle $C_e^i$ \augfifteen{of $K_n$} containing $e$.  Furthermore, 
$\Delta_e^1\cap \Delta_e^2=D[e]$.
\end{lemma}

\begin{figure}[ht]
\begin{center}
\includegraphics[scale=0.7]{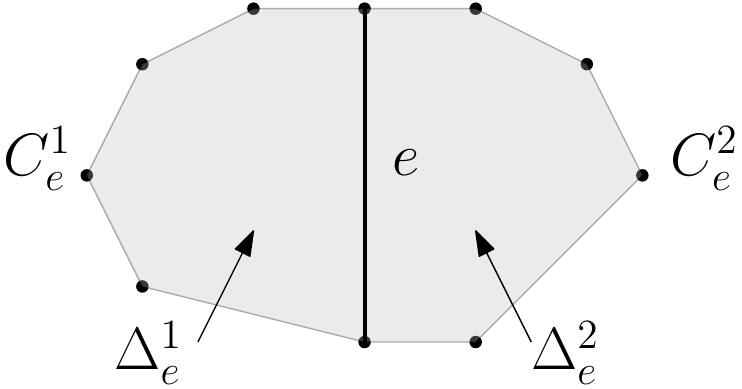} \label{fg:sides_delta}
\caption{Illustrating the consequence of Lemma \ref{lm:disjtDelta}.}
\end{center}
\end{figure}

\maytwonine{Figure \ref{fg:sides_delta} illustrates the conclusion of Lemma \ref{lm:disjtDelta}.  Before the proof, we give two simple consequences, which are also used in the next section.}

The \maytwonine{first} simple consequence is about edges not in either $\Delta_{e}^1$ or $\Delta_e^2$.
\marchTwentysix{Recall that $\sphere$ denotes the sphere.  \augfifteen{For $e\in E(K_n)$, set $F_e=\sphere\setminus (\Delta_{e}^1\cup\Delta_{e}^2)$. } }

\begin{lemma}\label{lm:crossingFi} Let $D$ be an h-convex drawing of $K_n$ with \deceight{a specified} witnessing set 
of convex sides and 
 $e$ and $e'$ be distinct edges of $K_n$.  If $D[e']$ has a point in $F_e$, then $e'$ has an end in each of $\Sigma_{e}^1$ and $\Sigma_{e}^2$.  
\end{lemma}

\begin{cproof} \marchTwentysix{Let $u$ and $v$ be the \julythirteen{ends of $e$}.} \janThree{We prove the contrapositive. Suppose that for some $k\in\{1,2\}$, $e'$ has both ends in $\{u,v\}\cup\Sigma_{e}^k$. Then $D[e']$ is contained in $\Delta_{e}^k$ by Lemma \ref{lm:disjtDelta}, so $D[e']\cap F_e=\varnothing$. }
\end{cproof}

We need some notation for the \decTwoOhTwoOh{next lemma}.
For distinct edges $e,e'$ of $K_n$, label each vertex of $C_{e}^1$ \julythirteen{(from Lemma \ref{lm:disjtDelta})} with 1, 2, 3, respectively, to indicate it is in $\Sigma_{e'}^1$, in $\Sigma_{e'}^2$, or incident \febFiveChange{with $e'$}.

\begin{lemma}\label{lm:noInterlace}  \mayfour{Suppose $D$ is an h-convex drawing of $K_n$ with \deceight{a specified} witnessing set 
of convex sides,  $e,e'$ are edges of $K_n$, and $C^1_{e}$ is labelled as in the preceding sentence.  Then t}here is no $1,2,1,2$ pattern in the cyclic order around $C_{e}^1$.  \end{lemma}

\begin{cproof}  Otherwise, there are four vertices $v_1,v_2,v_3,v_4$ \febSixteenChange{of $C_{e}^1$ in} this cyclic order  with $v_1,v_3$ having label 1 and $v_2,v_4$ having label 2.  As all  $v_i$ are in $\Sigma_{e}^1$, the definition of $C_{e}^1$ implies $v_1v_3$ crosses $v_2v_4$ in $D_{e}^1$.  

However, the edge $v_1v_3$ is in \febFiveChange{$D_{e'}^1$}, while $v_2v_4$ is in \febFiveChange{$D_{e'}^2$}. \maytwonine{Lemma \ref{lm:crossingFi} shows that both $v_1v_3$ and $v_2v_4$ are contained in $\Delta_e^1$, so they cross in $D$.  However,} Lemma \ref{lm:disjtDelta} implies they do not cross, a contradiction.
\end{cproof}

The remainder of this section is devoted to the proof of Lemma \ref{lm:disjtDelta}.  In the arguments below, we will use, without particular reference, the following observation:  if $e$ and $f$ are crossing edges, then each of the four 3-cycles in the unique $K_4$ containing $e$ and $f$ \julythirteen{has} a side (the one containing the fourth vertex of the $K_4$) that is definitely not convex.  Our drawings are convex, so such a 3-cycle has a unique convex side, which is in every set of convex sides witnessing h-convexity. \augfifteen{If $C$ is the 4-cycle bounding a face of this $K_4$, then the closed disc bounded by $D[C]$ and containing $D[e]\cup D[f]$ is the {\em crossing side\/} of $D[C]$.}

\begin{lemma}\label{lm:leftCycle}  Let $D$ be an h-convex drawing of $K_n$ with \deceight{a specified} witnessing set 
of convex sides, and let $e$ be an edge of $K_n$.
Let $i\in \{1,2\}$ and \febFiveChange{suppose $|\Sigma_e^i|\ge 1$}.   
\begin{enumerate}[label=(\roman*)] 
\item\label{it:eCrossFree} In $D_e^i$, $e$ is not crossed. In particular, there is a face $F_e^i$ \febSixteenChange{of $D_e^i$} incident with $e$ and containing \decTwoOhTwoOh{$\Sigma_e^{(3-i)}$}.  

\item\label{it:crossK4} \julythirteen{If $J$} is any crossing $K_4$ in $D_e^i$, then $F_e^i$ is contained in the face of $D[J]$ bounded by a 4-cycle.  In particular, no crossing of $D_e^i$ is incident with $F_e^i$, so $F_e^i$ is bounded by a cycle  in $D_e^i$.
\end{enumerate}
\end{lemma}

\begin{cproof}
If $e$ were crossed by an edge $f$ in $D_e^i$, then the crossing $K_4$ shows that the two ends of $f$ are \decTwoOhTwoOh{not both in the same one} of $\Sigma_e^1$ and $\Sigma_e^2$, a contradiction.   Thus, $e$ is not crossed and is incident with exactly two faces of $D_e^i$; $F_e^i$ is the one containing \decTwoOhTwoOh{$\Sigma_e^{(3-i)}$}, completing \ref{it:eCrossFree}.

For \ref{it:crossK4}, let $C$ be the 4-cycle in $J$ bounding a face of $D[J]$.  We rule out one trivial case immediately.  If $e$ is in $J$, then, since $e$ is \febTwelf{not crossed in $D_e^i$}, it is in $C$.  The \febSixteenChange{convex sides (these are unique and in $\mathcal C$)} of each 3-cycle in $J$ are all on the crossing side of $D[C]$.  \febSixteenChange{Since the two vertices of $J$ not incident with $e$ are in $\Sigma_e^i$,}  $F_e^i$ is contained in the \maytwonine{face of $D[J]$} bounded by $C$, as required.

Therefore, we may assume there is a vertex $u$ incident with $e$ and not in $J$.   As $e$ is not crossed in $D_e^i$, \febTwelf{$D[e]$} is contained in one of the faces $F$ of $D[J]$.  \febFiveChange{Since \febTwelf{$D[u]$} is incident with $F_e^i$,} $F$ is also the face of $D[J]$ containing~$F_e^i$. We are done if $F$ is bounded by a 4-cycle in $J$, so we assume\julythirteen{, by way of contradiction,} that $F$ is incident with the crossing of $D[J]$.

Convexity implies that the edges joining $u$ to the vertices of $J$ are all contained on the crossing side of $D[C]$.  Thus, $C$, $u$, and these four edges constitute a planar embedding of the 4-wheel $W$.  Each of the four 3-cycles in $W$ has its convex side on the crossing side of $D[C]$:  three of these 3-cycles are contained in convex sides of 3-cycles of $J$, so \febFiveChange{for them} the assertion follows from h-convexity; the fourth is in a crossing $K_4$ on the crossing side of $D[C]$.

If $e$ is one of the edges of $W$, then the \febSixteenChange{end of $e$ different from $u$ has} two neighbours in $C$ that are in different ones of $\Sigma_e^1$ and $\Sigma_e^2$, a contradiction. 
If $e$ is not in $W$, then its other end $v$ is in one of the four faces of $D[W]$ incident with $u$.  This implies that $v$ is on the convex side of the bounding 3-cycle and that the two vertices of $C$ in this 3-cycle \febFiveChange{are in different ones of $\Sigma_e^1$ and $\Sigma_e^2$, a contradiction.}
\end{cproof}

For $i=1,2$, \febSixteenChange{if $\Sigma_e^i\ne\varnothing$, then let} $C_e^i$ be the cycle in $D_e^i$ that is the boundary of $F_e^i$.  All its vertices not incident with $e$ are, \janThree{by definition}, in $\Sigma_e^i$.  \augfifteen{Note that, with $F_e$ \maytwonine{as defined} preceding Lemma \ref{lm:crossingFi}, $F_e=F_e^1\cap F_e^2$.}

Our next lemma asserts that, for any three vertices in $D_e^i$, the \deceight{specified} convex side that they bound is contained in the side of $D[C_e^i]$ not containing $F_e^i$.  In particular, this shows that $D_e^i$ is f-convex.

\augusttwenty{We remark that the main result of \cite{amrs} further implies that $D_e^i$ is pseudolinear.  Thus, \febFiveChange{any edge $e$} of an h-convex drawing partitions the vertices into two pseudolinear subdrawings \febFiveChange{$D_e^1$ and $D_e^2$}.  This generalizes the fact that, in a spherical drawing, \julythirteen{for each great circle $C$ that contains an edge, the vertices in either closed side of $C$ induce a \augfifteen{rectilinear drawing}}.}

\begin{lemma}\label{lm:SigmaIsF}
Let $D$ be an h-convex drawing of $K_n$ with \febSixteenChange{witnessing set $\{\Delta_T\mid T \textrm{ is a 3-cycle of } K_n\}$ of} convex sides.  Let $e$ be an edge of $K_n$ and let $i\in \{1,2\}$.  If $|\Sigma_e^i|\ge 1$, then, \febFiveChange{for each 3-cycle $T$ in $D_e^i$, $\Delta_T\cap F_e^i=\varnothing$.  In particular,} $D_e^i$ is f-convex.
\end{lemma}

\begin{cproof}
\julythirteen{Let $uvw$ be a 3-cycle 
 in $D_e^i$.}
 \ignore{the convex side of \febFiveChange{$\Delta_T$} is
 disjoint from $F_e^i$.}
Suppose \maytwonine{first that there is an edge $f$ with both ends in $C_e^i$ that crosses $uvw$}.  
\junesix{If $f$ has one end in $uvw$, then Lemma \ref{lm:leftCycle} \ref{it:crossK4} shows the crossing} in the $K_4$ that includes $u$, $v$, $w$, and $f$ is separated from $F_e^i$ by the face-boundary 4-cycle in the $K_4$.  Therefore, the convex side of each of the four 3-cycles \julythirteen{in the $K_4$ is the side that is disjoint from $F_e^i$.  In particular, this holds for $uvw$, as required.}

\julythirteen{In the remaining case, both} vertices incident with $f$ are in the side of $uvw$ that contains $F_e^i$.  \junesix{In this case, Definition \ref{df:convex} \ref{it:convexSide} of convex side shows it is the other side,} the one disjoint from $F_e^i$, that is convex, as required.
Therefore, we can assume no edge having both ends in $C_e^i$ crosses $uvw$.  

\janThree{Suppose that $C_e^i$ has at least four vertices, and let} $a$ and $b$ \augusttwenty{be} any two vertices of $C_e^i$ such that neither is an end of $e$. As $a$, $b$, and $e$ are all incident with $F_e^i$, the $K_4$ containing $a$, $b$, and $e$ has a face incident with all four of its vertices.  It follows that this is a crossing $K_4$. Lemma \ref{lm:leftCycle} implies that the crossing in this $K_4$ is separated from $F_e^i$ by the face-bounding 4-cycle.  Thus, all the 3-cycles in this $K_4$ have their convex side disjoint from $F_e^i$.  Although this was already known for the two 3-cycles containing $e$, we now know it holds for the two 3-cycles \augusttwenty{containing the edge $ab$}.

\janThree{Finally, if $C_e^i$ has only three vertices, then the result follows from h-convexity.}  Otherwise, let $y$ be one of the ends of $e$ and consider $C_e^i$ together with all the chords from $y$.  By the preceding paragraph, all of the 3-cycles using two of these \marchTwentysix{edges incident with $y$} and an edge of $C_e^i$ have their convex side disjoint from $F_e^i$.  From the earlier discussion, none of these chords crosses $uvw$.  It follows that $uvw$ is contained in the convex side of one of them; this convex side is disjoint from $F_e^i$.  Thus, the chosen convex side for $uvw$ is, by h-convexity, disjoint from $F_e^i$.\end{cproof}

The remaining detail about h-convex drawings we need is that $D_e^1\subseteq F_e^2$ and $D_e^2\subseteq F_e^1$.   \augfifteen{As mentioned after the proof Lemma \ref{lm:leftCycle}, $C_e^i$ is the boundary of the face $F_e^i$ of $D_e^i$.  The other closed disc in the sphere bounded by $D[C_e^i]$ is denoted $\Delta_e^i$.  Evidently, $D[\Sigma_e^i]\subseteq \Delta_e^i$.  We \janThree{begin by showing that} $\Sigma_e^2\cap \Delta_e^1=\varnothing$}.

\begin{lemma}\label{lm:no2VxInC1} Let $D$ be an h-convex drawing of $K_n$ with \deceight{a specified} witnessing set 
\febFiveChange{of convex sides}.  Let $e=uv$ be \augfifteen{an edge} of $K_n$ and $i\in \{1,2\}$.  \julythirteen{For $w\in V(G)\setminus\{u,v\}$, if $D[w]\subseteq \Delta_e^i$, then $w\in \Sigma_e^i$.} \end{lemma}

\begin{cproof} 
\janThree{Suppose that $D[w]\subseteq\Delta_e^1$.}
\janFourteen{If $C_e^1$ is a $3$-cycle, then its convex side  $\Delta_e^1$ contains $w$; h-convexity  implies that  
the $3$-cycle including $w$ and $e$  is contained in $\Delta_e^1$, and hence $w\in \Sigma_e^1$. Thus, we may assume that $C_e^1$ has length at least $4$.}

\janThree{ For each edge $ab\in C_e^1-e$, let $J_{ab}$ denote the crossing $K_4$ in $D$ induced by $e$ and $ab$.} 
\janThree{
The closed disc $\Delta_e^1$ is the union of the crossing sides of the $J_{ab}$, so $D[w]$ is contained in the crossing side of some $J_{ab}$. Since $J_{ab}$ is a crossing $K_4$, the convex sides of all the 3-cycles in $J_{ab}$ are determined.}

\janThree{If $D[w]$ is contained in the convex side of one of the 3-cycles $D[auv]$ or $D[buv]$, then it follows from h-convexity that this side contains the convex side of $D[wuv]$ that is in $\mathcal{C}$, and thus $w\in\Sigma_e^1$. Therefore we may assume that $D[w]$ is contained in the convex sides of both $D[abu]$ and $D[abv]$. Consequently, $D[wu]$ is contained in the convex side of $D[abu]$ and $D[wv]$ is contained in the convex side of $D[abv]$. The $K_4$ with vertices $a,w,u,v$ has a crossing in $D$ and determines the convex side of the 3-cycle containing $w$ and $e$, and therefore shows that $w\in\Sigma_e^1$.}
 \end{cproof}




Next we move on to edges.  The following result is preparatory to showing edges of $D_e^1$ and $D_e^2$ do not cross.

\begin{lemma}\label{lm:planarK4} Let $D$ be an h-convex drawing of $K_n$ with \deceight{a specified} \febSixteenChange{witnessing set of} convex sides.  
\febFiveChange{If,} for $i=1,2$, $x_i\in \Sigma_e^i$, then the 3-cycles induced by $x_1,e$ and $x_2,e$ do not cross in $D$. \end{lemma}

\begin{cproof}
Let $J$ be the $K_4$ induced by $x_1,x_2,e$ and, for $i=1,2$, let $T_i$ be the 3-cycle induced by $x_i,e$.  If $D[J-x_1x_2]$ has a crossing, then $T_1$ and $T_2$ cross but $e$ is not crossed.  \maytwonine{The contradiction is} that $x_1$ and $x_2$ are on the same side of $e$.  
\end{cproof}

We are now ready for the next major step towards the proof of Lemma \ref{lm:disjtDelta}.

\begin{lemma}\label{lm:no2edgeInC1} Let $D$ be an h-convex drawing of $K_n$ with \deceight{a specified} witnessing set of  convex sides. Then no edge of $D_e^2$ crosses any edge of $D_e^1$. \end{lemma}

\begin{cproof}  By way of contradiction, suppose some edge \janThree{$D[x_2y_2]$} of $D_e^2$ crosses some edge \janThree{$D[x_1y_1]$} of $D_e^1$.  Lemma \ref{lm:planarK4} implies not both $\{x_1,y_1\}$ and $\{x_2,y_2\}$ can contain an end of $e$.  Without loss of generality, we assume neither $x_1$ nor $y_1$ is an end of $e$.  Furthermore, $x_2y_2\ne e$, so we may choose the labelling such that $x_2$ is not an end of $e$.
Let $J_1$ be the $K_4$ induced by $x_1,y_1,e$.   

\janThree{Lemma \ref{lm:planarK4} implies that $x_2y_2$ does not cross any edge of $J_1$ incident with an end of $e$, so the only crossing of $x_2y_2$ with $J_1$ is with $x_1y_1$. Let $F$ be the face of $D[J_1]$ containing $F_e^1$. Lemma \ref{lm:no2VxInC1} shows that $D[x_2]\in F$, and similarly that $D[y_2]\in F$ if $y_2$ is not an end of $e$. As we traverse $D[x_2y_2]$ from $D[x_2]$, we cross $D[x_1y_1]$ once, and cross nothing else in $D[J_1]$. Therefore, $D[x_1y_1]$ is incident with $F$, as is $D[e]$. Since the face $F$ of $D[J_1]$ is incident with all four vertices of $J_1$, it follows that $J_1$ is a crossing $K_4$ in $D$.}

\janThree{Now, just after we traverse $D[x_2y_2]$ across $D[x_1y_1]$, we are in a face of $D[J_1]$ incident with the crossing of $D[J_1]$ and with both $D[x_1]$ and $D[y_1]$. This face is not incident with either end of $e$, nor is it equal to $F$. But, $y_2$ is either an end of $e$ or $D[y_2]$ lies in $F$, so $D[x_2y_2]$ must cross $D[J_1]$ a second time, which is a contradiction.}\end{cproof}






{We conclude} our study of h-convex drawings with the proof of Lemma~\ref{lm:disjtDelta}.

\begin{cproofof}{Lemma \ref{lm:disjtDelta}}    
\junesix{As in the paragraph immediately following the proof of Lemma \ref{lm:leftCycle}, for $i=1,2$, let $C_e^i$ be the cycle in $D_e^i$ that is the boundary of $F_e^i$.  Furthermore, let $\Delta_e^i$ be the closed disc bounded by $C_e^i$ that contains $D[\Sigma_e^i]$.}

\augusttwenty{The result is an application of the following simple fact about curves in the sphere.}  

\decTwoOhTwoOh{\begin{observation}\label{Observation} Let $\Delta_1$ and $\Delta_2$ be closed discs in the sphere bounded by simple closed curves $\gamma_1$ and $\gamma_2$, respectively. If: 
\begin{itemize} \item  $\gamma_1\cap \gamma_2$ is an arc (or empty); and \item $\gamma_1\not\subseteq \Delta_2$ and $\gamma_2\not\subseteq \Delta_1$,
\end{itemize}
then $\Delta_1\cap \Delta_2=\gamma_1\cap \gamma_2$.
\end{observation}}

\augusttwenty{Lemmas \ref{lm:no2VxInC1} and \ref{lm:no2edgeInC1} imply that, for $\{i,j\}=\{1,2\}$, the open arc \decTwoOhTwoOh{$D[C_e^i]\setminus D[e]$} is disjoint from $\Delta_e^j$.  The result is an immediate application of Observation \ref{Observation}.}
\end{cproofof}

\section{Proof of Theorem \ref{th:main}}\label{sec:hCxHasExactExtension}

In this section, \decTwoOhTwoOh{we prove \ref{it:Dhc}$\Rightarrow$\ref{it:Dps}:} an h-convex drawing of $K_n$ has simple closed curve extensions of the edges satisfying \ref{it:psVertex}, \ref{it:psCrossings}, and \ref{it:psEdge}.  
\decTwoOhTwoOh{This completes the proof of Theorem \ref{th:main}.}

The proof  iteratively constructs the \octthirteen{set  of} simple closed curve extensions of the edges. We assume that, \octthirteen{for some $J\subset E(K_n)$ and for all $e\in J$, there exist extensions $\gamma_e$ satisfying (PS1), \ref{it:psCrossings}, (PS3), and a fourth property (PS4):}
\begin{enumerate}[leftmargin=\widthof{\textbf{(PS4)}}+\labelsep+\labelsep]
\item[\textbf{(PS4)}] \janThree{For \maytwonine{each $e\in J$}} \janFourteen{and each  $e'\in E(K_n)\setminus\{e\}$}, \janThree{ $\gamma_e$ intersects} \janFourteen{the}  \janThree{closed edge} \janFourteen{$D[e']$} \janThree{at most once, and, if it exists, the point of intersection is either a crossing or a vertex incident with both $e$ and $e'$.}
\end{enumerate}
\janThree{Notice that if $J=E(K_n)$, then the extensions of the edges in $J$ automatically satisfy (PS4) provided they satisfy (PS1), \ref{it:psCrossings}, (PS3).} \janFourteen{The extra assumption (PS4) is required for inductive purposes.}


\janThree{\octthirteen{We pick any $e_0\in E(K_n)\setminus J$; the extension of $\gamma_{e_0}$ is obtained as a result of Theorem \ref{thm:sweeping}.  Thus, we need to find the two initial curves $\gamma_{e_0}^1$ and $\gamma_{e_0}^2$ satisfying the hypotheses of Theorem \ref{thm:sweeping} with respect to \deceight{$\{\gamma_e\mid e\in J\}$}.  The curve $\gamma_{e_0}^i$ contains $D[e_0]$ and is completed by an arc joining the ends of $e_0$ that \maythreeone{is in $F_{e_0}$} and ``very near'' the path $C_{e_0}^i-e_0$\deceight{; this is where we use the specified convex sides  of an h-convex drawing}. Both curves are contained in $D[e_0]\cup F_{e_0}$.} How ``near'' is ``very near'' will depend on the curves that are already determined. Our next lemma is the crucial point; Corollary \ref{co:existFirstCurve} provides $\gamma_{e_0}^1$ and $\gamma_{e_0}^2$.}



\begin{lemma}\label{lm:onlyTwoIntersections} Let $D$ be an h-convex drawing of $K_n$ with witnessing set $\mathcal C$ of convex sides.
Let $J\subseteq E(K_n)$ \mayfour{and suppose that, for each \aprilEight{$e\in J$}, there is a simple closed curve $\gamma_e$ in $D[e]\cup F_e$} \julythirteen{containing $D[e]$}, \janThree{and such that the extensions $\{\gamma_e\mid e\in J\}$ satisfy \deceight{\ref{it:psVertex} and} (PS4).}
\mayfour{If \augusttwenty{$e_0\in E(K_n) \setminus J$}}, then\octthirteen{, for any \deceight{$i\in\{1,2\}$ and any} sufficiently small neighbourhood $N$ of \deceight{$D[C_{e_0}^i]$} \deceight{in $F_{e_0}\cup D[C_{e_0}^i]$}, there are at most two arcs in \deceight{$(\gamma_e\cap N)\setminus D[C_{e_0}^i]$}.}
\augfifteen{Furthermore, at most one of these segments is contained in $D[e]$.}
\end{lemma}

\begin{cproof}  
We begin with the central claim. \deceight{For the sake of definiteness, we assume $i=1$.}

\begin{claim}\label{cl:uAndArc}
\augfifteen{If $u$ is \janThree{a vertex} incident with $e$ and is in $C_{e_0}^1$, then there is no arc of $\gamma_e$ contained in the interior of $\Delta_{e_0}^1$ that joins two points in $D[C_{e_0}^1]$ neither of which is} \janThree{$D[u]$.}
\end{claim}

\begin{proof}\marchTwentysix{
Suppose to the contrary that there is such an arc $\alpha$. \deceight{It follows from $\alpha\subseteq \gamma_e$ and (PS4) that} there is no closed edge of $D[C_{e_0}^1]$ containing both ends $y,z$ of $\alpha$. Thus, each component of $D[C_{e_0}^1]\setminus \{y,z\}$ has a vertex of $C_{e_0}^1$.  Let \janThree{$D[w]$} be a vertex in the component of \aprilEight{$D[C_{e_0}^1]\setminus\{y,z\}$} not containing \janThree{$D[u]$}.}

\marchTwentysix{Both \febTwelf{$u$} and \febTwelf{$w$} are in $C_{e_0}^1$, so both are \febTwelf{drawn} in $\Delta_{e_0}^1$.  Therefore, $D[uw]\subseteq \Delta_{e_0}^1$, so $D[uw]$ crosses $\alpha$.  Since $\alpha\subseteq \gamma_e$, \janThree{$D[e]$} cannot cross $\alpha$, so $uw\ne e$.  But the edge \janThree{$D[uw]$} has the two points \janThree{$D[u]$} and the crossing \augfifteen{with $\alpha$ in  $\gamma_e$}, \maythreeone{contradicting (PS4) and completing the proof.}
}\end{proof}

\janThree{Suppose that $\gamma_e\cap D[C_{e_0}^1]$} \julythirteen{contains a vertex \janThree{$D[u]$, which must be incident with $e$ since $\gamma_e$ satisfies \ref{it:psVertex}}. If both ends of $e$ are in $C_{e_0}^1$, then \janThree{$D[e]\subseteq D_{e_0}^1$ since $D_{e_0}^1$ is the subdrawing of $D$ induced by} \janFourteen{$\Sigma_{e_0}^1$}. \janThree{Furthermore,} \maythreeone{the claim}  implies \janThree{that the ends of $e$} are the only intersections of $\gamma_e$ with $C_{e_0}$.}  
Therefore \janThree{$D[e]$} \augfifteen{is} \julythirteen{the only segment of $\gamma_e$ contained in $\Delta_{e_0}^1$.}

\janThree{We may therefore assume that} $u$ is the only end of \janThree{$e$} in $C_{e_0}^1$.  Then there are only two directions from \janThree{$D[u]$} in $\gamma_e$; each of these directions can give an arc in $\gamma_e\cap \Delta_{e_0}^1$ having \janThree{$D[u]$} as an end.  By Claim \ref{cl:uAndArc}, these are the only possible intersections of $\gamma_e$ with $\Delta_{e_0}^1$. \febTwelf{In this case, $\gamma_e\cap\Delta_{e_0}^1$ is either one arc, with $D[u]$ as either an end or an interior point, or $\gamma_e\cap\Delta_{e_0}^1$ \maythreeone{is just} $D[u]$.}


\julythirteen{It follows that if an end of $e$ is in $C_{e_0}^1$, then the two cases above show that \janFourteen{$\gamma_e\cap \Delta_{e_0}^1$ is  either}
\febTwelf{a non-trivial arc or $D[u]$.} \janFourteen{In the former case,}
 only the ends of this arc can be the start of a segment} \febTwelf{of $\gamma_e$} \janFourteen{from a point of \janThree{$D[C_{e_0}^1]$} into $F_{e_0}$, as required.} \janFourteen{In the latter case, $\gamma_e\cap \Delta_{e_0}^1=D[u]$, and there are exactly two arcs of $\gamma_e$ having an end in $D[u]$ and extending into $F_{e_0}$.}


Thus, we may assume that no point of $\gamma_e\cap D[C_{e_0}^1]$ is a vertex.  In this case, \deceight{(PS4) implies that} every intersection of $\gamma_e$ with \janThree{$D[C_{e_0}^1]$} is a \janFourteen{crossing. }
\janFourteen{ To} \janThree{prove that there are at most two segments of $\gamma_e$ from a point of $D[C_{e_0}^1]$ into $F_{e_0}$,} \julythirteen{it suffices to prove that $\gamma_e$ has at most two crossings with \janFourteen{$D[C_{e_0}^1]$}.}

\marchTwentysix{Suppose by way of contradiction that there are three crossings of $\gamma_e$ with \janThree{$D[C_{e_0}^1]$}.}   
\marchTwentysix{Traverse $D[C_{e_0}^1]$ in one direction from such a crossing $z$.  The first vertex \janThree{$D[v]$} reached is incident with an edge \aprilEight{$f$} of $C_{e_0}^1$ such that both \janThree{$D[v]$} and $z$ are in the closed edge \aprilEight{$D[f]$}.  Since $z\in \gamma_e$, none of the rest of \aprilEight{$D[f]$} (including \janThree{$D[v]$}) is in $\gamma_e$.  In particular, \janThree{$D[v]$} is in $\Sigma_{e}^1\cup \Sigma_{e}^2$. }

\maythreeone{Suppose \janThree{$D[w]$} is} the first vertex reached from $z$ traversing $D[C_{e_0}^1]$ in the other direction, \janThree{that is, if $w$ is the other vertex incident with $f$}\maythreeone{. Because $\gamma_e\subseteq D[e]\cup F_e$,} \janThree{$D[v]$} and \janThree{$D[w]$} are on different sides of $\gamma_e$.  Therefore, they are in different ones of $\Sigma_{e}^1$ and $\Sigma_{e}^2$.   Thus, every crossing of $\gamma_e$ with \janThree{$D[C_{e_0}^1]$} produces a change between 1 and 2 in the ``1,2,3"-labelling of Lemma \ref{lm:noInterlace}.  Let $z_1,z_2,z_3$ be three crossings of $\gamma_e$ with $D[C_{e_0}^1]$, in this cyclic order.  Then the vertices of $C_{e_0}^1$ nearest each $z_k$ have labels 1 and 2.  Starting at $z_1$, we find 1 and 2 near it.  Up to relabelling, we may assume the 1 occurs between $z_1$ and $z_3$ and the 2 between $z_1$ and $z_2$.  Then choose the 1 near $z_2$ and the 2 near $z_3$ to obtain a 1,2,1,2 pattern, contradicting Lemma \ref{lm:noInterlace}. \end{cproof}

The following corollary is a straightforward consequence of Lemma \ref{lm:onlyTwoIntersections}.

\begin{corollary}\label{co:existFirstCurve}\julythirteen{Let $D$ be an h-convex drawing of $K_n$ with witnessing set $\mathcal C$ of convex sides.
Let $J\subseteq E(K_n)$ and suppose that, for each \aprilEight{$e\in J$}, there is a simple closed curve $\gamma_e$ in $D[e]\cup F_e$ containing $D[e]$}, \janThree{and such that the extensions $\{\gamma_e\mid e\in J\}$ satisfy \ref{it:psVertex} and (PS4).}
For any \deceight{$i\in\{1,2\}$ and any} sufficiently small neighbourhood $N$ of \deceight{$D[C_{e_0}^i]$ in $F_{e_0}\cup D[C_{e_0}^i]$}, there is a choice of \deceight{$\gamma_{e_0}^i$} in $N$ such that
\febTwelf{the curves in \deceight{$J\cup\{\delta_{e_0}^i\}$}
 satisfy \ref{it:psVertex}, \ref{it:psWeakCrossings}, \ref{it:psEdge}, and (PS4). \hfill\eop
 }
\end{corollary}

We are now ready for the proof of Theorem \ref{th:main}.

\begin{cproofof}{\decTwoOhTwoOh{Theorem \ref{th:main}  \ref{it:Dhc} implies \ref{it:Dps}}} \octeight{\julythirteen{\augfifteen{Suppose $J\subseteq E(K_n)$} and we have, for each $e\in J$, a simple closed curve $\gamma_e$, such that the set $\{\gamma_e\mid e\in J\}$ satisfies \deceight{\ref{it:psVertex}--\janThree{(PS4)}}. 
If $J=E(K_n)$, then we are done; otherwise, let $e_0\in E(K_n)\setminus J$.} 
}

\octeight{
\julythirteen{We show there is a \augfifteen{curve $\gamma_{e_0}$} containing $D[e_{0}]$ and \febFiveChange{otherwise in} \augfifteen{the face $F_{e_0}$} of \febFiveChange{$D_{e_{0}}^1\cup D_{e_{0}}^2$} bounded by \febFiveChange{$(C_{e_{0}}^1-e_{0})\cup (C_{e_{0}}^2-e_{0})$} and \augfifteen{such that $\{\gamma_ e\mid e\in J\cup\{e_0\}\}$} satisfies} \deceight{\janThree{\ref{it:psVertex}--\janThree{(PS4)}.}}  
}


\octeight{
 Let $M$ consist of \maytwonine{those $e\in E(K_n)\setminus (J\cup \{e_0\})$} such that $D[e]\cap F_{e_0}\ne \varnothing$.  In any order, repeatedly use Corollary \ref{co:existFirstCurve} to obtain, for all $e\in M$, $\delta_e^1$ so that the curves \augfifteen{in the set $\Gamma=\{\gamma_e\mid e\in J\}\cup \{\delta_e^1\mid e\in M\}$} \janThree{satisfy \ref{it:psVertex}, \deceight{\ref{it:psWeakCrossings}, \ref{it:psEdge}, and \janThree{(PS4)}.}} 
 }
 
 \octeight{
 For each \augfifteen{$e\in J\cup M$, \marchTen{$\gamma_e\setminus e$}} is in the face \janThree{$F_e$} of $D_{e}^1\cup D_{e}^2$ bounded by $(C^1_{e}-e)\cup (C^2_{e}-e)$, $\Sigma_{e}^1$ is on one \augfifteen{side of $\gamma_e$} and $\Sigma_{e}^2$ is on the other side \augfifteen{of $\gamma_e$}.  \julythirteen{Let $\delta_{e_0}^1$ be as in Corollary \ref{co:existFirstCurve} with respect to $e_0$ and $J$; this is our first approximation to $\gamma_{e_0}$.} Applying Corollary \ref{co:existFirstCurve}  to $C_{e_0}^2$, consider  an analogous curve $\delta_{e}^2$. By Corollary  \ref{co:existFirstCurve}, both $\Gamma\cup \{\delta_{e_0}^1\}$ and $\Gamma\cup \{\delta_{e_0}^2\}$ satisfy \ref{it:psVertex},  
\janThree{\decTwoOhTwoOh{\ref{it:psEdge}, (PS4)}. Moreover, $\delta_{e_0}^1$ and $\delta_{e_0}^2$ intersect each curve in $\Gamma$  at most twice and all intersections are crossings.}
}
 
 \octeight{
 For $i=1, 2$, let $\Gamma_i=\{\delta\in \Gamma\;:\; \delta\cap \delta_{e_0}^i\neq \emptyset\}$. 
 }

 \begin{claim}
 \octeight{
 Either $\Gamma=\Gamma_1\cup \Gamma_2$ or $D$ is f-convex.}
\end{claim}  
\begin{proof}
\octeight{
Suppose that for some $e\in J\cup M$ the extension $\delta\in \Gamma$ of $e$  is not in $\Gamma_1\cup \Gamma_2$. From  Lemma \ref{lm:crossingFi}  it follows that any edge of $M$ crosses both $\delta_{e_0}^1$ and $\delta_{e_0}^2$. Therefore  $e\in J$ and  $\delta=\gamma_{e}$.  
 Since  $\gamma_e$ does not intersect $\delta^1_{e_0}\cup \delta^2_{e_0}$, we conclude that $\gamma_e$ is disjoint from $F_{e_0}$.
 }

Recall that, for \marchTen{$\ell=1,2$, $\Delta^\ell_{e_{0}}$} is the closed disc in \marchTen{$D^\ell_{e_{0}}$} bounded by \marchTen{$C^\ell_{e_{0}}$} and disjoint from $F_{e_0}$. The preceding paragraph implies that, for some \augfifteen{$\ell\in\{1,2\}$, $\gamma_e$} is contained in \marchTen{$\Delta_{e_{0}}^\ell$}.  It follows that:  (i) $e$ has both ends in \marchTen{$\Delta_{e_{0}}^\ell$}; and (ii) every vertex of \marchTen{$C_{e_{0}}^\ell$} is in the same one of $\Delta_{e}^1$ and $\Delta_{e}^2$ (because, by \augfifteen{assumption, $\gamma_e$} separates $\Delta_{e}^1$ from $\Delta^2_{e}$).

If an edge $xy$ crosses $e$, then $x$ and $y$ are \decTwoOhTwoOh{not both in the same one} of $\Sigma_{e}^1$ and $\Sigma_{e}^2$.  Therefore, (ii) implies that, if $x,y$ are vertices in \marchTen{$C_{e_{0}}^\ell$}, then $xy$ does not cross $e$.  In particular, letting $z$ be one end of $e_{0}$ and letting $xy$ run through the edges of \marchTen{$C_{e_{0}}^\ell$}, the 3-cycles $xyz$ bound convex sides that cover \marchTen{$\Delta_{e_{0}}^\ell$}.  It follows that $e$ is contained in one of these; let it be $xyz$.  

Suppose by way of contradiction that $e$ has an end $u$ that is not one of $x,y,z$.  Then h-convexity implies that the convex sides in $\mathcal C$ of the 3-cycles $uxy$, $uxz$, and $uyz$ are all contained in the convex side of $xyz$.  

Let $v$ be the other end of $e$.  If $v\in\{x,y,z\}$, then the planar $K_4$ induced by $u,x,y,z$ shows that the \augusttwenty{two vertices in $\{x,y,z\}\setminus \{v\}$} \decTwoOhTwoOh{are not both in the same one of} $\Sigma_{e}^1$ and $\Sigma_{e}^2$, a contradiction.  Likewise, if $v\not\in\{x,y,z\}$, then the the 3-cycle in $\{uxy,uxz,uyz\}$ that contains $v$ in its convex side has its two vertices from $x,y,z$ \decTwoOhTwoOh{in different ones} of $\Sigma_{e}^1$ and $\Sigma_{e}^2$, a contradiction.  These contradictions show that both ends of $e$ are among $x,y,z$; that is, both ends of $e$ are in \marchTen{$C_{e_{0}}^\ell$}.

Next, suppose by way of contradiction, that $e$ is not an edge of \marchTen{$C_{e_{0}}^\ell$}.  Then it is a chord of \marchTen{$C_{e_{0}}^\ell$ in $D_{e_{0}}^\ell$}, and so it crosses an edge $xy$ with $x$ and $y$ in \marchTen{$C_{e_{0}}^\ell$} on different sides of $e$.  But then we have $x$ and $y$ are \decTwoOhTwoOh{in different ones} of $\Sigma_{e}^1$ and $\Sigma_{e}^2$, a contradiction.

Lemma \ref{lm:SigmaIsF} shows that \marchTen{$D_{e_{0}}^\ell$} (using the convex sides in $\mathcal C$) is f-convex.  Since $e$ is in \marchTen{$C_{e_{0}}^\ell$}, it follows that \augusttwenty{the} vertices of \marchTen{$D_{e_{0}}^\ell$} not incident with $e$ are\augusttwenty{, for some $k\in\{1,2\}$, all in $\Sigma_{e}^k$}.  Since \marchTen{$\delta_j\subseteq \Delta_{e_{0}}^\ell$}, all vertices of \marchTen{$\Delta_{e_{0}}^{3-\ell}$} are \augusttwenty{in the same $\Sigma_{e}^k$ as the two vertices incident with} $e_{0}$.  It follows that all vertices of $K_n$ \augusttwenty{not incident with $e$} are in the same \augusttwenty{$\Sigma_{e}^k$}, showing that $D$ is f-convex, as claimed.

\end{proof}

  In the case $D$ is f-convex, \decTwoOhTwoOh{\cite[Thm.~1]{amrs}} shows that $D$ is homeomorphic to a pseudolinear drawing in the plane.  By definition, the pseudolines intersect once in the plane, and they \janThree{can be chosen so that they} all cross again at the point at infinity that completes the sphere.  
  
  \octeight{
  Thus we may assume that $\Gamma_1\cup \Gamma_2=\Gamma$. In this case we apply Theorem \ref{thm:sweeping} to 
  $\Gamma$, $\sigma=D[e_0]$,   $\gamma_1=\delta_{e_0}^1$ and $\gamma_2=\delta_{e_0}^2$ to obtain a curve $\gamma_{e_0}$ \augfifteen{such that $\{\gamma_ e\mid e\in J\cup\{e_0\}\}$} satisfies \ref{it:psVertex}, \ref{it:psCrossings}, \janone{\ref{it:psEdge}, and \janThree{(PS4)}},  as desired.
  }
\end{cproofof}

\deceight{\octthirteen{A given h-convex \julythirteen{drawing $D$} may have different choices for the convex sides of the 3-cycles that witness h-convexity.  \augustsix{In \junesix{Section \ref{sec:HcxIsPS}, the extensions} of $D$ into arrangements of pseudocircles rely on a choice of convex sides witnessing h-convexity. Moreover, the proof of the implication \ref{it:Dwps}${}\Rightarrow{}$\ref{it:Dhc} shows how the choice of convex sides can be recovered from such an arrangement of pseudocircles. \octthirteen{Define two such arrangements of pseudocircles to be {\em equivalent\/} if they determine the same convex sides.
}}}
}

\deceight{\octthirteen{The sweeping theorem of Snoeyink and Hershberger \cite{sweeping} shows that either of two equivalent arrangements of pseudocircles can be shifted to the other by a sequence of Reidemeister II and III moves.  Simple examples show that Type II moves may be required.}
}

\section{Non-extendible drawings of $K_9$ and $K_{10}$}\label{sec:K9andK10}

\octthirteen{In this section, we present a drawing of each of $K_9$ and $K_{10}$.  The \decTwoOhTwoOh{drawing $D_9$} of $K_9$ in Figure \ref{fg:Kten} has an extension to an arrangement of pseudocircles \deceight{(that is, \ref{it:psWeakCrossings})} that satisfies \ref{it:psVertex}, but no such extension also satisfies \ref{it:psCrossings}.  The \decTwoOhTwoOh{drawing $D_{10}$} of $K_{10}$ in Figure \ref{fg:Kten} does not have an extension \deceight{to an arrangement of pseudocircles}.}

\begin{figure}[ht]
\begin{center}
\includegraphics[scale=0.55]{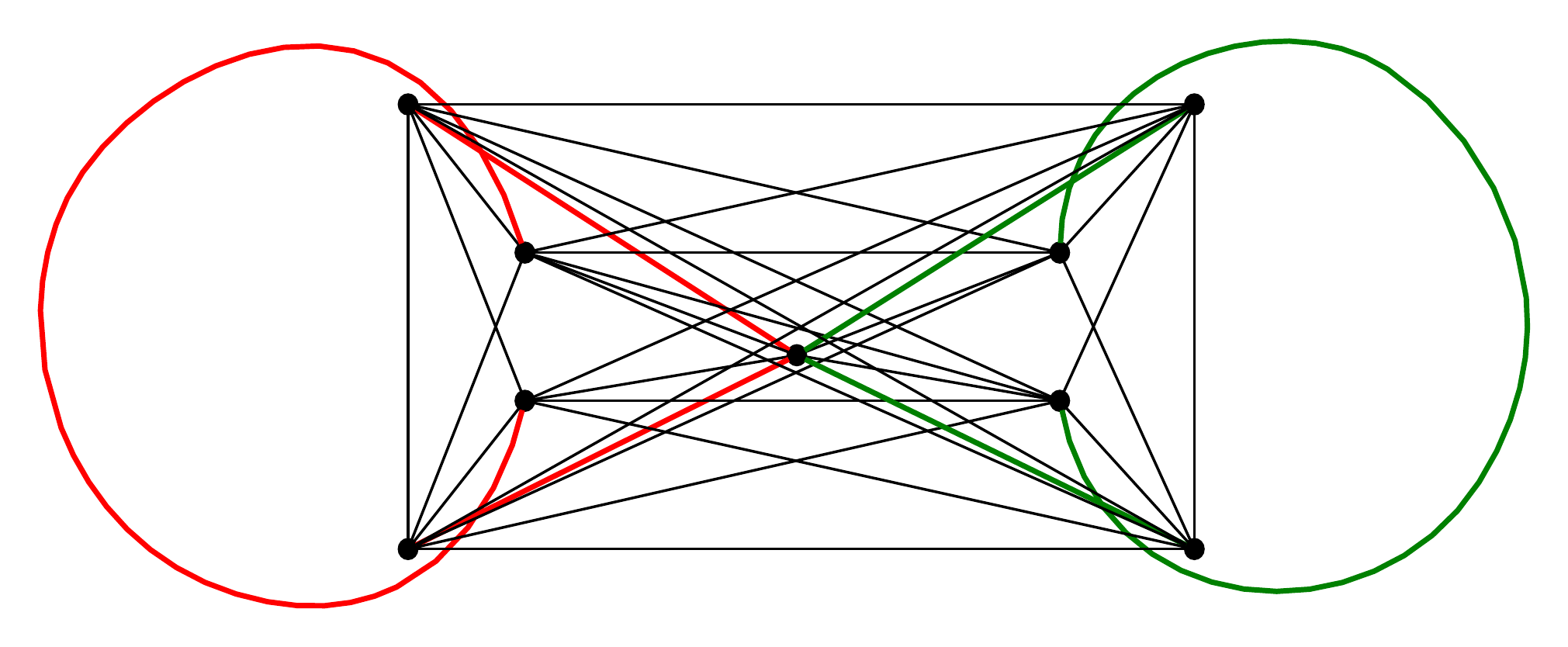}\hfill\includegraphics[scale=0.75]{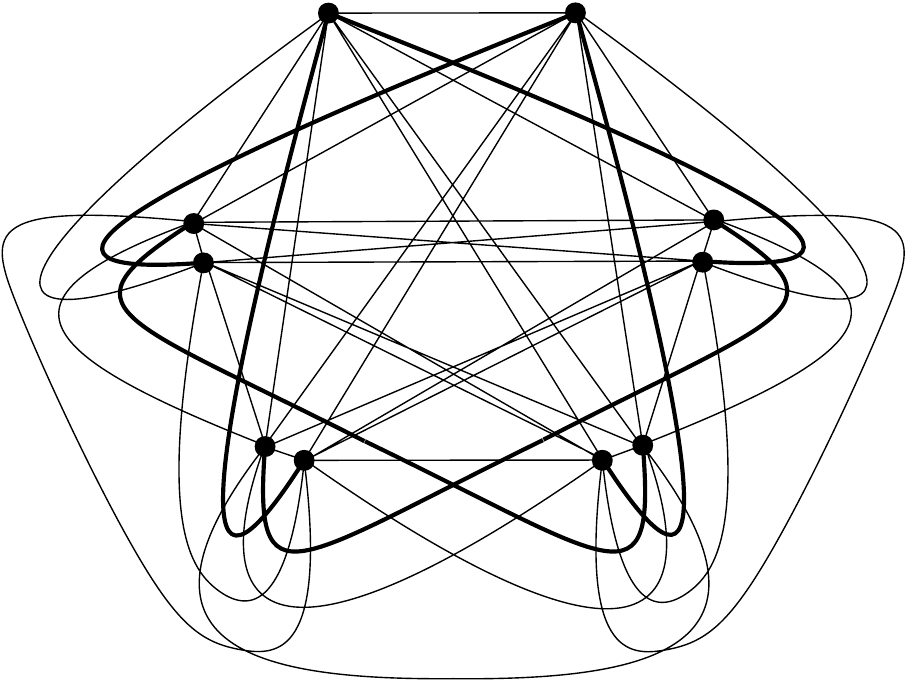}
\caption{The upper \decTwoOhTwoOh{drawing $D_9$} of $K_9$ has a pseudocircular extension, but none satisfying \ref{it:psCrossings}.  \decTwoOhTwoOh{The coloured subdrawings in $D_9$ are both isomorphic to $D_1$ (see Figure \ref{cex1} below)}.  The lower \decTwoOhTwoOh{drawing $D_{10}$} of $K_{10}$ has no pseudocircular extension.   \decTwoOhTwoOh{In the drawing of $K_{10}$, the thicker edges are the drawing $D_2$ (Figure \ref{cex1}).}}
\label{fg:Kten}
\end{center}
\end{figure}

\deceight{The analyses of  the drawings $D_9$ and $D_{10}$ involve spirals in the drawings $D_1$ and $D_2$ in Figure \ref{cex1}.  For convenience, we extend the notion of spiral to closed spiral.  For an arrangement $\Gamma$ of simple closed curves, a {\em closed spiral\/} \decTwoOhTwoOh{is a} simple closed curve $\gamma$ in $P(\Gamma)$ with basepoint $s$ such that, for every sufficiently small open interval $I$ in $\gamma$ containing $s$, $\gamma\setminus I$ is a spiral.
}

\deceight{To see the point of closed spirals, consider the drawing $D_1$.  \janone{The unique simple closed curve $\gamma$ in $D_1$ is a closed spiral, whose basepoint is the vertex in $\gamma$.  }Let $\Gamma$ be an arrangement of pseudocircles extending $D_1$.  Then, for any sufficiently small open interval $I$ in $\gamma$ containing $s$, $I$ does not contain any crossing of $\Gamma$ other than $s$, and $\gamma\setminus I$ is a spiral in $P(\Gamma)$.  By Theorem \ref{th:technical}, $\gamma\setminus I$ has an external segment, which corresponds to a simple closed curve in $P(\Gamma)$ that is contained in the interior of $\gamma$.
}

\remove{\deceight{[Text removed.]}}  It follows that, any extension \decTwoOhTwoOh{of $D_9$ to} an arrangement of pseudocircles must have a pseudocircle in the bounded side of each of the \decTwoOhTwoOh{coloured} copies of $D_1$, showing it does not satisfy \ref{it:psCrossings}.

\begin{figure}[ht!]
\centering
\decTwoOhTwoOh{\includegraphics[scale=0.5]{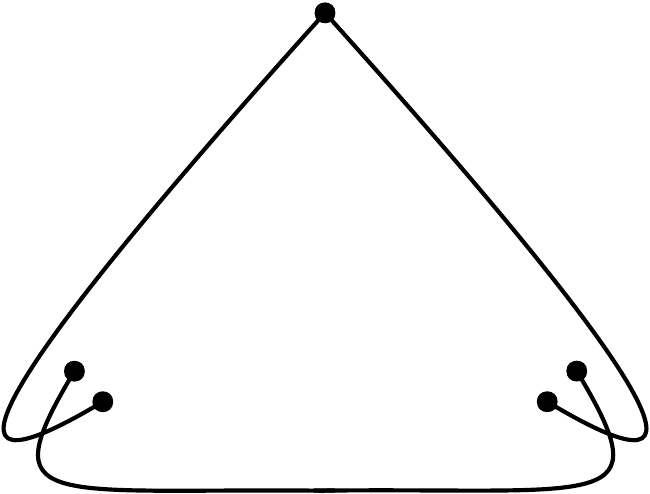}}\hskip 1.5truein
\includegraphics[scale=0.5]{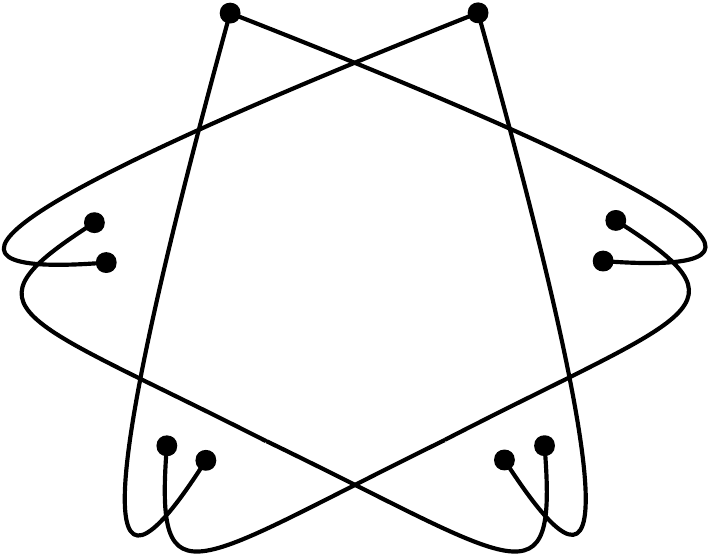}
\caption{The crucial configurations $D_1$ on the left and $D_2$ on the right.}
\label{cex1}
\end{figure}

\octthirteen{On the other hand, we can extend the straight edges into lines and extend the two curved edges by line segments connecting their vertices. In the sphere, adding the point at infinity to the straight lines gives an extension of the drawing of $K_9$ to an arrangement of simple closed curves satisfying \ref{it:psVertex} \decTwoOhTwoOh{and \ref{it:psWeakCrossings}}. Some of the pseudocircles may intersect tangentially at the point at infinity since we did not assume that the lines are in general position, but this can be corrected if we perturb the curves.   Thus, \janone{without \ref{it:psEdge},  \ref{it:psVertex} and \ref{it:psWeakCrossings}} do not imply \ref{it:psVertex} and \ref{it:psCrossings}.
}

\octthirteen{For the \decTwoOhTwoOh{drawing $D_{10}$}, we show \decTwoOhTwoOh{that there} is no arrangement of pseudocircles extending the drawing $D_2$ on the right in Figure \ref{cex1}.  Since $D_2$ is contained \decTwoOhTwoOh{in  $D_{10}$}, this implies there is no arrangement of pseudocircles \decTwoOhTwoOh{extending  $D_{10}$}.
}

\octthirteen{Let $\gamma$ denote the unique simple closed curve in $D_2$ containing the \deceight{ten crossings and none of the vertices.  Let $s$ be the upper most crossing in the diagram.  Then $\gamma$ is a closed spiral.}}

\deceight{For any sufficently small open interval $I$ in $\gamma$ containing $s$,  $\gamma\setminus I$ is a spiral in $P(\Gamma)$ that has weight 7 with decomposition $\alpha_0\alpha_1\dots\alpha_{7}$}.  The drawing $D_2$ already shows that the segments $\alpha_1,\dots,\alpha_7$ are coherent.  The segments $\alpha_0$ and $\alpha_7$ are symmetric.  

\octthirteen{The extension $\alpha_0^+$ is contained in  pseudocircle $\gamma_0$ containing $\alpha_0$.  We note that $\gamma_0$ contains a vertex outside of our original simple closed curve $\gamma$.  Therefore, the choice of $I$ shows $\alpha_0^+$ has its end on $\gamma\setminus I$.  Consequently, $\alpha_0$ is also coherent.  Likewise, $\alpha_7$ is coherent, showing $P(\Gamma)$ contains a coherent spiral, contradicting Theorem \ref{th:technical}.  That is, $D_2$ and, consequently, the drawing of $K_{10}$ do not have extensions to arrangements of pseudocircles.
}

\section{Conclusion}\label{sec:conclusion}

In this section, we mention a few results and questions about drawings of $K_n$ related to this work.    We start with the questions.

\marchTen{The \decTwoOhTwoOh{drawing $D_{10}$} in Figure \ref{fg:Kten} \decTwoOhTwoOh{has no} extension to \janone{an arrangement of pseudocircles}.  It is natural to wonder if there is a fixed $k$ such that every drawing of $K_n$ has an extension \octthirteen{to an arrangement of simple closed curves} that pairwise cross at most $k$ times.  In fact, for a drawing of any simple graph, there is such an extension with $k\le 4$.  See Figure \ref{fg:fourCrossings} \octthirteen{for an idea of how such an extension may be achieved}.}

\begin{figure}[ht!]
\begin{center}
\includegraphics[scale=.1]{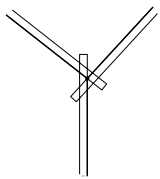}
\caption{\marchTen{Extensions near a vertex.}}
\label{fg:fourCrossings}
\end{center}
\end{figure}

Some interesting questions remain unresolved.  

\begin{enumerate}[label={\bf Question \arabic*.},leftmargin = 1.10 truein, ref=Question \arabic*]


\item Does every convex drawing of $K_n$ have an extension to simple closed curves pairwise crossing exactly twice, or if not, at most twice?  \marchTen{The drawings in Figure \ref{fg:Kten} are not convex.}

\item Arroyo et al \cite{abr} characterize drawings of (not necessarily complete) graphs whose edges extend to an arrangement of pseudolines by giving the complete (infinite) list of obstructions.  Given the close connection we developed here between \deceight{pseudospherical} and pseudolinear drawings for complete graphs, it is reasonable to \julythirteen{wonder if there is an analogous theorem for arrangements of pseudocircles.}  

We conjecture that there is a list-of-obstructions characterization of when an arbitrary graph has an extension satisfying \ref{it:psVertex}, \ref{it:psCrossings} and \ref{it:psEdge}.  It is not clear to us at this juncture how to proceed with this.  \remove{\deceight{[Text removed.]}}

\end{enumerate}

The study of spirals and Theorem \ref{th:technical} played an essential role in the proof of Theorem \ref{th:main} and also led to the drawings $D_1$ and $D_2$ used in the drawings of $K_9$ and $K_{10}$. 

\deceight{Recall that a string is a homeomorph of a compact real interval.}

\begin{enumerate}[label={\bf Question \arabic*.},leftmargin = 1.10 truein, ref=Question \arabic*, start=3]
\item{Can we characterize \deceight{those} sets of strings  that are extendible to an arrangement of pseudocircles?}
\end{enumerate}

\maytwonine{So far, the authors have not found 
an example of a set of strings not-extendible to a set of pseudocircles
that cannot be explained in terms of spirals and Theorem \ref{th:technical}. A further study of spirals may be the key to solve this question.}

We conclude with two simple results about pseudospherical drawings of $K_n$.  For the first, Lemma \ref{lm:disjtDelta} implies that every edge of a pseudospherical drawing of $K_n$ induces a split of the drawing into two pseudolinear drawings of smaller complete graphs (one having, say, $k$ vertices and the other $n+2-k$ vertices). Every pseudolinear drawing of $K_n$ has at least $n^2+{}$O$(n\log(n))$ empty triangles (that is, 3-cycles having a side that does not contain a vertex) \cite{amrs}.  Thus, we can estimate the number of empty triangles on one side or other of the split.  Adding in empty triangles involving vertices on different sides of the split yields at least $\frac34n^2+ {}$O$(n\log(n))$ empty triangles in a pseudospherical drawing of $K_n$.

Rafla \cite{rafla} conjectured that every (good) drawing of $K_n$ has a Hamilton cycle \febFiveChange{with no self-crossing}.   \'Abrego et al.\ have enumerated all the drawings of $K_n$ with $n\le 9$ \cite{oswin2} and in this way verified the conjecture for all these drawings of $K_n$.  \decTwoOhTwoOh{We extend to pseudospherical drawings the folklore proof that a pseudolinear drawing of $K_n$ has such a Hamilton cycle.}

\decTwoOhTwoOh{For each edge $e$ of the pseudospherical drawing $D$ of $K_n$, let $c(e)$ denote the number of pseudocircles that $D[e]$ crosses.  
Choose a Hamilton cycle $H$ in $K_n$ that minimizes $\sum_{e\in E(H)}c(e)$.  
If $e_1,f_1\in E(H)$ cross in $D$, then let $J$ be the $K_4$ induced by $e_1$ and $f_1$.  
Then $D[J]$ has exactly one crossing, namely $e_1$ with $f_1$.  
The remaining four edges in $J$ come in two disjoint, non-crossing pairs.  
For one such pair $\{e_2,f_2\}$, 
$(H-\{e_1,f_1\})+ \{e_2,f_2\}$
 is also a Hamilton cycle $H'$.}

\decTwoOhTwoOh{We claim that $\sum_{e\in E(H')}c(e)<\sum_{e\in E(H)}c(e)$, contradicting the choice of $H$.  \begin{enumerate}\item Clearly, the two contributions of $e_1$ crossing $\gamma_{f_1}$ and $f_1$ crossing $\gamma_{e_1}$ are counted for the $H$-sum, but not for the $H'$-sum.  
\newline\indent Now suppose $e\in E(H')$ crosses the pseudocircle $\gamma_f$ containing $f\in E(H')$.  \item If  $e$ is not $\{e_2,f_2\}$, then the $e\gamma_f$-crossing is counted for both $H$ and $H'$.  \item If $e=e_2$, say, then $\gamma_f$ crosses $e_2$ and must continue through another side of the \rbr{4-cycle $D[J-\{e_1,f_1\}]$}.  In order to get there, it must cross at least one of $e_1$ and $f_1$.  If $\gamma_f$ crosses both $e_2,f_2$, then it also crosses both $e_1,f_1$.  Therefore, $\gamma_f$ has at least as many crossing with $e_1,f_1$ that it has with $e_2,f_2$, as required.\end{enumerate}}

\end{document}